\documentclass{amsart}       
\usepackage{graphicx}
\usepackage{mathptmx}      

\usepackage{amsmath,amssymb}
\usepackage{thmtools}
\usepackage{thm-restate}
\usepackage{braket}

\usepackage{hyperref}

\usepackage{tikz}
\usetikzlibrary{tikzmark, matrix, decorations.markings, calc, arrows}
\usepackage{tikz-cd}

\usepackage[noend]{algpseudocode}
\algrenewcommand\algorithmicindent{2.0em}
\usepackage{algorithm}
\usepackage{multicol}

\usepackage{enumitem}
\setlist{parsep=0pt,listparindent=\parindent}
\usepackage{fancybox}

\newtheorem{definition}{Definition}
\newtheorem{theorem}{Theorem}
\newtheorem{lemma}{Lemma}
\newtheorem{example}{Example}
\newtheorem{notation}{\textbf{\underline{Notation}}}

\usepackage{printlen}
\newcommand{\A}{\mathbb{A}}

\newcommand{\vp}{\varphi}
\newcommand{\AR}{Auslander-Reiten } 
\newcommand{\mat}[1]{\left[#1\right]}

\DeclareMathOperator{\End}{End}
\DeclareMathOperator{\Hom}{Hom}
\DeclareMathOperator{\rep}{rep}

\newcommand{\arr}[1]{\operatorname{arr}(#1)}

\newcommand{\n}{\emptyset}
\renewcommand{\*}{*}
\newcommand{\0}{0}
\newcommand{\E}{E}

\newcommand{\SNF}{ \begin{smallmatrix} E & 0 \\ 0 & 0 \end{smallmatrix} }

\newcommand{\smat}[1]{ \left[\begin{smallmatrix} #1 \end{smallmatrix}\right] }

\newcommand{\Id}{\mathrm{id}}

\newcommand{\KQ}{K\!Q}
\newcommand{\CL}{C\!L}
\newcommand{\intv}{\mathbb{I}}
\newcommand{\reltoeq}{\trianglerighteq}
\newcommand{\relto}{\vartriangleright}

\newcommand{\dimv}[2]{\raisebox{1ex}{$\overset{#1}{\underset{#2}{}}$}}

\algnewcommand{\LineComment}[1]{\State \(//\) #1}
\algnewcommand{\IfThenElse}[3]{
  \State \algorithmicif\ #1\ \algorithmicthen\ #2\ \algorithmicelse\ #3}
\algnewcommand{\ForDo}[2]
  { 
  \State \algorithmicforall\ #1\ \algorithmicdo\ #2
  }

\makeatletter
\def\Int #1{\expandafter\Int@i#1\@nil}
\def\Int@i #1,#2\@nil{{#1{:}#2}}

\def\IntC #1{\expandafter\IntC@i#1\@nil}
\def\IntC@i #1,#2\@nil{{#1{,}#2}}

\def\itoi #1{\expandafter\itoi@i#1\@nil}
\def\itoi@i #1,#2,#3,#4\@nil{_{\Int{#1,#2}}^{\Int{#3,#4}}}
\makeatother

\makeatletter
\DeclareRobustCommand{\rvdots}{%
  \vbox{
    \baselineskip4\p@\lineskiplimit\z@
    \kern-\p@
    \hbox{.}\hbox{.}\hbox{.}
  }}
\makeatother

\makeatletter
\def\tikz@delimiter#1#2#3#4#5#6#7#8{%
  \bgroup
    \pgfextra{\let\tikz@save@last@fig@name=\tikz@last@fig@name}%
    node[outer sep=0pt,inner sep=0pt,draw=none,fill=none,anchor=#1,at=(\tikz@last@fig@name.#2),#3]
    {%
      {\nullfont\pgf@process{\pgfpointdiff{\pgfpointanchor{\tikz@last@fig@name}{#4}}{\pgfpointanchor{\tikz@last@fig@name}{#5}}}}%
      \delimitershortfall\z@
      \resizebox*{!}{#8}{$\left#6\vcenter{\hrule height .5#8 depth .5#8 width0pt}\right#7$}%
    }
    \pgfextra{\global\let\tikz@last@fig@name=\tikz@save@last@fig@name}%
  \egroup%
}
\makeatother

\newcommand{\mattikz}[1]{
   \begin{tikzpicture}[
      baseline = (p.center), 
      ampersand replacement=\&,
      decoration={
        markings,
        mark=
        at position 0.5
        with{
          \draw[-] (-2pt,-2pt) -- (2pt,2pt);
          \draw[-] (2pt,-2pt) -- (-2pt,2pt);
        }
      }]
      {#1}
    \end{tikzpicture}
}

\newcommand{\matheader}{
  \matrix[matrix of math nodes, column sep=0mm, row sep=0mm,
  inner sep = 0mm, left delimiter = {[},right delimiter = {]},
  every node/.append style={
    anchor=center,text depth = 0.375em,text
    height=0.875em,minimum width=1.25em}](p)
}

\newcommand{\rowempty}{
  \matrix[matrix of math nodes, column sep=0mm, row sep=0mm,
  inner sep = 0mm, left delimiter = {[},right delimiter = {]},
  nodes in empty cells,
  every node/.append style={
    anchor=center,text depth = 0em,text
    height=3mm, minimum width=1.8em}](p)
  {
    \vphantom{.}\\
  };
}

\newcommand{\rowarrow}[2]{
  \draw[->, red, postaction={decorate}] #1  to[out=30,in=330] #2;
}

\newcommand{\colarrow}[2]{
  \draw[->, red, postaction={decorate}] #1 to[out=240,in=300] #2;
}

\title{Matrix Method for Persistence Modules on Commutative Ladders of Finite Type}

\author{Hideto Asashiba}
\address{
  \begin{flushleft}
    Hideto Asashiba\\
    Department of Mathematics, Faculty of Science, Shizuoka University
  \end{flushleft}
}
\email{asashiba.hideto@shizuoka.ac.jp}

\author{Emerson G. Escolar}
\address{
  \begin{flushleft}
    Emerson G. Escolar\\
    Center for Advanced Intelligence Project, RIKEN
  \end{flushleft}
}
\email{emerson.escolar@riken.jp}

\author{Yasuaki Hiraoka}
\address{
  \begin{flushleft}
    Yasuaki Hiraoka\\
    Kyoto University Institute for Advanced Study, Kyoto University\\
    Center for Advanced Intelligence Project, RIKEN   
  \end{flushleft}
}
\email{hiraoka.yasuaki.6z@kyoto-u.ac.jp}

\author{Hiroshi Takeuchi}
\address{
  \begin{flushleft}
    Hiroshi Takeuchi\\
    Mathematics Department, Graduate School of Science, Tohoku University\\
    Research Fellow of Japan Society for the Promotion of Science
  \end{flushleft}
}
\email{hiroshi.takeuchi.s6@dc.tohoku.ac.jp}

\date{}

\thanks{This is a pre-print of an article published in Japan Journal of Industrial and Applied Mathematics. The final authenticated version is available online at: \url{https://doi.org/10.1007/s13160-018-0331-y}.}

\begin{document}

\begin{abstract}
  {
    The theory of persistence modules on the commutative ladders
  $\CL_n(\tau)$ provides an extension of persistent homology. However, an efficient
  algorithm to compute the generalized persistence diagrams is still
  lacking.
  In this work, we view a persistence module $M$ on $\CL_n(\tau)$ as a morphism between zigzag modules, which can be expressed in a block matrix form.
  For the representation finite case ($n\leq 4$), we provide an algorithm that uses certain permissible row and column operations to compute a normal form of the block matrix. In this form an indecomposable decomposition of $M$, and thus its persistence diagram, is obtained.
}
\end{abstract}

\keywords{Persistence modules,  Commutative ladders, Computational topology, Algorithms}
\subjclass[2000]{68W30, 16G20, 55N99}

\maketitle

\section{Introduction}
\label{sec:intro}

Recently, the paper \cite{PMCL} introduced the study of persistence
modules on the commutative ladders of finite type. This was
motivated in part by a need to study simultaneously robust and common
topological features using the ideas of persistent homology
\cite{topo_pers}. Let us first give an overview of this
background and motivation.

One way to construct persistent homology is the following. Let $\mathbb{X}$ be a \emph{filtration}, a non-decreasing sequence of spaces
\[
  \mathbb{X}: X_1 \subset X_2 \subset \hdots \subset X_n.
\]
Applying a homology functor
$H(-)$ with coefficient field $K$, we obtain a sequence
\begin{equation}
  \label{eq:PH}
	H(\mathbb{X}) \colon H(X_1) \rightarrow H(X_2) \rightarrow \cdots \rightarrow H(X_n)
\end{equation}
of $K$-vector spaces and induced linear maps between them, called the
persistent homology of the filtration.

Diagram \eqref{eq:PH} above can be interpreted in the language of the
representation theory of (bound) quivers. This leads one to considering
persistence modules in general, of which $H(\mathbb{X})$ in Diagram~\eqref{eq:PH} is one
example. With this point of view, a persistence module can be taken to
be synonymous to a representation of a bound quiver.

Assuming that $H(X_i)$ is finite dimensional for $i\in\{1,\hdots,n\}$, it is known that the persistence module $H(\mathbb{X})$ can be decomposed into the so-called \emph{interval representations}. The decomposition into intervals can be
used to study the persistent, robust, or multiscale topological features in $\mathbb{X}$. The length of each interval (its lifetime) can be interpreted as a measure of persistence or robustness of the topological feature.

More generally, different classes of persistence modules may be used to
study, using similar ideas, diagrams of spaces that are not filtrations.
As an example, zigzag persistent homology \cite{zigzag} can be used
to analyze common topological features in a collection of spaces. Here, let us consider the following simple example of zigzag
persistence. Given two spaces $X$ and $Y$, we can form the diagram
\[
\mathbb{X}:
\begin{tikzcd}
X \rar{\subset} & X \cup Y & Y \lar[swap]{\supset}
\end{tikzcd}.
\]
Applying $H(-)$, we obtain the diagram
\begin{equation}
  \label{eq:zigzagunion}
  H(\mathbb{X}):
  \begin{tikzcd}
    H(X) \rar{} & H(X \cup Y) & H(Y) \lar[swap]{}
  \end{tikzcd}
\end{equation}
of homology vector spaces and induced linear maps.
Similar to the classical persistent homology case, it is known that a
zigzag module, for example $H(\mathbb{X})$ in
Diagram~\eqref{eq:zigzagunion}, can be decomposed into interval zigzag modules.
Those that are nonzero from the left (at $X$), through the middle, and to the right (at $Y$) correspond to topological features that are
common to $X$ and $Y$.

A shortcoming of the above is that only robust features or only common
features can be studied, but not both simultaneously. A motivation for
using \emph{persistence modules on commutative ladders} \cite{PMCL} is
to deal with simultaneously common and robust topological features.
This can be thought of as a partial generalization towards
multidimensional persistence \cite{multipers}.

Let us review how we use commutative ladders to treat simultaneously
common and robust features. Suppose that $X_1 \subset X_2$ and $Y_1
\subset Y_2$ are two-step filtrations of spaces $X$ and $Y$. To study
the robust and common features shared between them, form the following
commutative diagram of homology vector spaces and linear maps:
\begin{equation}
\label{eq:comm_union}
	\begin{tikzcd}
		H(X_{2}) \rar & H(X_{2}\cup Y_{2}) \rar[leftarrow] & H(Y_{2}) \\
		H(X_{1}) \rar \uar & H(X_{1}\cup Y_{1}) \rar[leftarrow] \uar & H(Y_{1}) \uar
	\end{tikzcd}
\end{equation}
where the linear maps are induced from the respective inclusions.
In this diagram, the vertical direction captures the robust features,
while the horizontal direction captures the common features between
$X$ and $Y$. Indecomposable direct summands isomorphic
to
\[
  \begin{tikzcd}
    K \rar{1} & K \rar[leftarrow]{1} & K \\
    K \rar{1}\uar{1} & K \rar[leftarrow]{1}\uar{1} & K\uar{1}
  \end{tikzcd}
\]
if any, represent the simultaneously robust and common features.

The above discussion provides some motivations for our interest in persistence modules.
In this work, we shall not discuss what particular class of spaces and which homology functor $H(-)$ are to be used. Instead, we take a persistence module as our starting point. In particular, we consider persistence modules on the commutative ladders $\CL_n(\tau)$, which we
define in Section~\ref{subsec:cl}. Diagram~\eqref{eq:comm_union} is an
example of a persistence module on the bound quiver
\begin{equation} \label{lfb}
  \CL_3(fb):
  \begin{tikzcd}[ampersand replacement = \&]
    \circ \rar[""{name=U1}] \& \circ \arrow[r,-,""{name=U2}]\& \circ\lar \\ 
    \circ \rar[""{name=D1}] \uar \& 
    \circ \uar\arrow[-,r,""{name=D2}] 
    \arrow[phantom, to path = {(U1) -- node[pos = 0.7, scale = 1.5] {$\circlearrowleft$}(D1)}]{}\& 
    \circ \lar\uar 
    \arrow[phantom, to path = {(U2) -- node[pos = 0.7, scale = 1.5] {$\circlearrowleft$}(D2)}]{}
  \end{tikzcd}.
\end{equation}
As in classical persistence,  an indecomposable decomposition of a
persistence module plays a key role in understanding its different
types of persistent topological features. In the general case however, the indecomposable
summands are not completely given by intervals or analogues of intervals.

The algorithm provided in \cite{PMCL} computes an indecomposable
decomposition by performing changes of bases on the individual vector
spaces in a given persistence module and extracting direct summands.
This involves working with the persistence module by its collection of
linear maps.

Here, we take a different point of view and reconsider a persistence
module on a commutative ladder as a morphism from its bottom row to
its top row, via Theorem~\ref{isothm} in Subsection~\ref{subsec:repn_to_arrow}.
Note that the bottom and top rows are nothing but zigzag
modules, and thus can be decomposed into interval zigzag modules. Using this fact, the
morphism can be written in a block matrix form with respect to these decompositions. In
essence, we treat the persistence module as one matrix, but with
certain restrictions induced from the structure of homomorphism spaces
between interval zigzag modules. We make these ideas precise in
Subsections \ref{subsec:matrixformalism} and
\ref{subsec:permissibles}.

We then provide a procedure for computing an indecomposable
decomposition using the above described matrix formalism. The idea is
to use column and row operations, as in elementary linear algebra, to
find normal forms. While the matrix has entries given by
homomorphisms between zigzag modules, the procedure can be
reinterpreted to involve only $K$-matrices, provided certain
restrictions on the permissible operations on the matrices are
respected. These restrictions are also derived from the structure of
the homomorphism spaces between the intervals.

The procedure is formalized in Algorithm~\ref{algo:main} in Section 4.2.2. The main theorem of this paper is the following.
\begin{restatable}{theorem}{algomainthm}
  \label{main_theorem}
  Assume Algorithm~\ref{algo:main}
  is called with the block matrix problem corresponding to a persistence
  module $M$ on a commutative ladder of finite type. Then
  Algorithm~\ref{algo:main} terminates and the input matrix is
  transformed to an isomorphic block matrix consisting only of
  identity, zero, and strongly zero blocks, and whose indecomposable
  decomposition corresponds to an indecomposable decomposition of $M$.
\end{restatable}

Finally, we note that our problem of computing a normal
form of a block matrix under certain permissible operations falls
under a more general class of problems called ``matrix problems''.
Matrix problems can be given a theoretical framework via the
representation theory of bocses \cite{boevey,rojter}. In this framework, the matrix reductions can be interpreted as reduction functors that induce equivalences of representation categories of bocses.
In this work, however, we have kept the necessary
theoretical background to a minimum and expressed
Algorithm~\ref{algo:main} in terms of block matrices and permissible operations.


\section{Background}
\subsection{Quivers and Persistent Homology}
\label{subsec:quiverph}

A \emph{quiver} $Q = (Q_0, Q_1)$ is a directed graph with set of
vertices $Q_0$ and set of arrows $Q_1$. An arrow $\alpha \in Q_1$ from
a vertex $a \in Q_0$ to a vertex $b \in Q_0$ is denoted by $\alpha
\colon a \rightarrow b$. In this case, $a$ is called the source of
$\alpha$, and $b$ is its target. A \emph{path} $p = (a \mid \alpha_1
\dots \alpha_\ell \mid b)$ of length $\ell$ from a vertex $a$ to a
vertex $b$ is a sequence of $\ell$ arrows $\alpha_1, \dots,
\alpha_\ell$, where the source of $\alpha_1$ is $a$, the target of
$\alpha_\ell$ is $b$, and the target of $\alpha_i$ is equal to the
source of $\alpha_{i + 1}$ for all $i \in \{1,\hdots,\ell-1\}$. Note that paths of length $0$
are allowed. These are the paths $e_a
= (a||a)$, called the stationary path at $a$, for each vertex $a$. Moreover, for each arrow $\alpha:a\rightarrow b$, we use the same symbol to denote the corresponding path $\alpha = (a\mid\alpha\mid b)$.

Let $K$ be a field, which we fix throughout this work. The
\emph{path algebra} $\KQ$ of a quiver $Q$ is the following
$K$-algebra. As a $K$-vector space, it is freely generated by all
paths in $Q$. The multiplication in $\KQ$ is defined by setting
\[
	(a \mid \alpha_1 \cdots \alpha_\ell \mid b)(c \mid \beta_1 \cdots \beta_m \mid d) = 
  \begin{cases}
    (a \mid \alpha_1 \cdots \alpha_\ell \beta_1 \cdots \beta_m \mid d) & {b = c}, \\
    0 & \text{otherwise;} 
  \end{cases}
\]
and extending $K$-linearly. In this work, we consider only finite
quivers $(|Q_0|, |Q_1| < \infty)$ without any oriented
cycles\footnote{An oriented cycle is a path with nonzero length whose
  source is equal to its target.}. With this assumption, $\KQ$ is a
finite-dimensional $K$-algebra.

Let $\{w_1,\hdots, w_m\}$ be a finite set of $m$ paths that share a
common source $s\in Q_0$ and a common target
$t\in Q_0$. A linear combination
\begin{equation*}
	\rho = \sum_{i = 1}^{m}{c_i w_i} \in \KQ, \text{ where } c_i \in K,
\end{equation*}
is called a \emph{relation} in $Q$.

A bound quiver $(Q,P)$ is a pair of a quiver $Q$ together with a set
of relations $P = \{\rho_i\}_{i\in T}$. The two-sided ideal of $\KQ$
generated by a set of relations $P = \{\rho_i\}_{i\in T}$ is
denoted by $\langle P \rangle$. The algebra of a bound quiver $(Q,P)$
is the quotient $A = \KQ/\langle P \rangle $.

A \emph{representation of a quiver $Q$}, denoted $M = (M_a,
\varphi_{\alpha})_{a \in Q_0, \alpha \in Q_1}$, is a collection of a
finite dimensional vector space $M_a$ for each $a \in Q_0$ and a
linear map $\varphi_{\alpha} \colon M_a \rightarrow M_b$ for each arrow
$\alpha \colon a \rightarrow b$. 

Let $M = (M_a, \varphi_{\alpha})_{a \in Q_0, \alpha \in
  Q_1}$ be a representation $Q$, and $w = (a \mid \alpha_1 \cdots
\alpha_\ell \mid b)$ a path in $Q$. Define the evaluation of $M$ on
the path $w$ to be $\vp_w = \vp_{\alpha_\ell} \circ \dots \circ \vp_{\alpha_1}: M_a \rightarrow M_b$.
The representation $M$ is said to be a \emph{representation of a bound quiver} $(Q,P)$ if
$\varphi_{\rho} \doteq \sum\limits_i{c_{i} \varphi_{w_{i}}} = 0$
for all relations $\rho = \sum\limits_i{c_{i} w_{i}} \in P$.

For example, let $Q$ and $M$ be the following quiver and representation:
\begin{equation} 
  \label{cl2}
	Q:
	\begin{tikzcd}
		\overset{3}{\circ} \rar[rightarrow]{\alpha} &
    		\overset{4}{\circ} \\
		\underset{1}{\circ} \rar[rightarrow]{\gamma} \uar[rightarrow]{\beta} &
		\underset{2}{\circ} \uar[rightarrow,swap]{\delta}
	\end{tikzcd}
  \text{ and }
  M:
	\begin{tikzcd}
		M_{3} \rar[rightarrow]{\vp_\alpha} &
    		M_{4} \\
		M_1 \rar[rightarrow]{\vp_\gamma} \uar[rightarrow]{\vp_\beta} &
		M_2 \uar[rightarrow,swap]{\vp_\delta}
	\end{tikzcd}
\end{equation}
respectively. If $P = \{ \rho = \gamma \delta - \beta\alpha \}$, then $M$ is a representation of
$(Q,P)$ if and only if $\vp_\rho = \vp_\delta \vp_\gamma - \vp_\alpha \vp_\beta =
0$. In other words, this implies that $M$ in Diagram~\eqref{cl2} forms a commutative diagram of
$K$-vector spaces and linear maps. In general, we define the set of
\emph{commutative relations} $C$ of a quiver $Q$ to be the set of
relations of the form $p - p'$ where $p$ and $p'$ are any two different
paths from vertices $a$ to $b$, for any pair of vertices $a$ and $b$.

\begin{definition}
The \emph{representation category} $\rep Q$ of $Q$ is the following category.
\begin{itemize}
\item Objects: finite-dimensional representations of the quiver $Q$.
\item Morphisms: Let $M = (M_a, \varphi_{\alpha})_{a \in Q_0, \alpha
    \in Q_1}$ and $N = (N_a, \psi_{\alpha})_{a \in Q_0, \alpha \in
    Q_1}$ be representations of $Q$. A morphism $f:M \rightarrow N$ is
  a collection of $K$-linear maps $f_a: M_a \rightarrow N_a$ such that
  for all arrows $\alpha:a\rightarrow b$ in $Q$, the diagram
  \begin{equation}
    \label{eq:repn_hom}
    \begin{tikzcd}
      M_a \rar{\varphi_\alpha} \dar{f_a} & M_b \dar{f_b} \\
      N_a \rar{\psi_\alpha} & N_b
    \end{tikzcd}
  \end{equation}
  is commutative. The collection of morphisms from $M$ to $N$ is
  denoted by $\Hom(M,N)$.
\item Composition: for $f = \{f_a\}_{a\in Q_0} \colon
  V \rightarrow W$ and $g = \{g_a\}_{a\in Q_0} \colon
  W \rightarrow U$, $g \circ f =
  \{g_af_a\}_{a\in Q_0}$.
\end{itemize}
\end{definition}
The representation category $\rep(Q,P)$ of a bound quiver $(Q,P)$ is
the full subcategory of $\rep{Q}$ with objects consisting of the
representations of $(Q,P)$.

The \emph{direct sum} $M \oplus N$ of representations
$M = (M_a, \varphi_{\alpha})$
and
$N = (N_a, \psi_{\alpha})$
of $(Q,P)$ is the representation with the vector space $M_a \oplus N_a$ for each vertex
$a \in Q_0$ and the linear map $\varphi_{\alpha} \oplus
\psi_{\alpha} \colon M_a \oplus N_a \rightarrow M_b \oplus N_b$
for each arrow $\alpha \colon a \rightarrow b$.

A representation $M \neq 0$ is
\emph{indecomposable} if $M \cong N \oplus N'$ implies $N = 0$ or $N' =
0$. From the Krull-Remak-Schmidt theorem, every representation $M$ can
be decomposed into a sum of indecomposable representations $M \cong W_1
\oplus \dots \oplus W_s$, unique up to isomorphism and permutation of terms. A quiver $Q$ or a
bound quiver $(Q,P)$ is said to be \emph{finite type}
(\emph{representation-finite}) if the number of isomorphism classes of
its indecomposable representations is finite, and is \emph{infinite type}
(\emph{representation-infinite}) otherwise.
For more details on the representation theory, see for example \cite{blue}.

Let $f$ and $b$ be symbols, representing ``forward'' and
``backward''. An \emph{orientation} $\tau$ is a
sequence $\tau = (\tau_1,\hdots,\tau_{n-1})$ where $\tau_i$ is
either $f$ or $b$ for each ${1 \leq i \leq n-1}$. Given $n\geq 1$ and an orientation $\tau$, define the quiver
\[
	\A_n(\tau) \colon
	\begin{tikzcd}
		\overset{1}{\circ} \rar[leftrightarrow] & 
    \overset{2}{\circ} \rar[leftrightarrow] & 
    \cdots \rar[leftrightarrow] & 
    \overset{n}{\circ},
	\end{tikzcd}
\]
where the $i$-th arrow
$\overset{i}{\circ}\longleftrightarrow\overset{i+1}{\circ}$ is
$\longrightarrow$ if $\tau_i = f$, and is $\longleftarrow$ if
$\tau_i = b$. We say that a quiver $\A_n(\tau)$ is $\A_n$-type.

From Gabriel's theorem \cite{Gabriel}, any $\A_n$-type quiver is
representation-finite. For $1\leq b\leq d \leq n$, define the
\emph{interval representation}
\[
	\intv[b, d] \colon 0 \longleftrightarrow \cdots
 \longleftrightarrow 0
  \longleftrightarrow \overset{b\text{-th}}{K} \longleftrightarrow K
  \longleftrightarrow \cdots
 \longleftrightarrow \overset{d\text{-th}}{K} 
 \longleftrightarrow 0 \longleftrightarrow \cdots
  \longleftrightarrow 0,
\]
in $\rep{\A_n(\tau)}$, which consists of copies of the vector space $K$ from indices $b$ to $d$ and $0$ elsewhere, and where the maps between the vector spaces $K$ are identity
maps and zero otherwise. It is known that $\left\{\intv[b,d]\right\}_{1\leq b
  \leq d \leq n}$ gives a complete list of indecomposable
representations of $\A_n(\tau)$ up to isomorphism. Thus, any representation $M$
of $\A_n(\tau)$ can be decomposed as a direct sum
\begin{equation} 
\label{eq:decompo_an}
M \cong \bigoplus\limits_{1 \leq b \leq d \leq n} \intv[b, d]^{m_{b, d}}
\end{equation}
where the numbers $m_{b, d} \in \mathbb{Z}_{\geq0}$ are multiplicities.

Classical persistent homology can be viewed as a representation (a
\emph{persistence module}) $M$ of $\A_n(\tau)$ with the orientation $\tau
= ff \cdots f$. Each interval representation $\intv[b,d]$ that appears
as a direct summand in a given persistence module tracks a homology
class which is born in $H(X_b)$ and persists up to $H(X_d)$. The
lengths of these intervals can be taken as encoding the persistence or
robustness of the homological features of the filtration. 

The \emph{persistence diagram} $D_M$ of a persistence module $M$ on
$\A_n(\tau)$ is the multiset
\begin{equation*}
	D_M = \set{(b, d) \text{ with  multiplicity } m_{b, d} | 1 \leq b \leq d \leq n},
\end{equation*}
where the multiplicities $m_{b,d}$ are determined by an indecomposable decomposition of $M$ as in Eq.~\eqref{eq:decompo_an}.
The persistence diagram can be visualized by plotting
the points $(b,d)$ together with multiplicities $m_{b,d}$ on a plane, and provides a compact way to represent
the presence and lifespans of the persistent topological features.

The ideas of persistent homology have been extended to a wide variety of underlying quivers.
For example, consider a collection $X_1, X_2, \dots, X_M$ of
spaces $X_j$ that do not form a filtration. Instead, one can form the diagram
\[
	\begin{tikzcd}
		X_1 \rar[hookrightarrow] &
		X_1 \cup X_2 \rar[hookleftarrow] &
		X_2 \rar[hookrightarrow] &
    X_2 \cup X_3 \rar[hookleftarrow] &
		\cdots \rar[hookleftarrow] &
		X_M
	\end{tikzcd}	
\]
and obtain the
persistence module
\begin{equation} \label{zz}
	\begin{tikzcd}[column sep=1.5em]
		H(X_1) \rar[rightarrow] &
		H(X_1 \cup X_2) \rar[leftarrow] &
		H(X_2) \rar[rightarrow] &
    H(X_2 \cup X_3)  \rar[leftarrow] &
		\cdots  \rar[leftarrow] &
		H(X_M)
	\end{tikzcd}
\end{equation}
which is a representation of $\A_n(fbfb\cdots fb)$.

Since the underlying quiver is $\A_n$-type, the indecomposable representations are given by the
intervals. An indecomposable decomposition of Eq.~\eqref{zz}, gives the persistent homological features in the collection $X_1, \dots, X_M$. In this case, the interval
representations can be interpreted as features common among certain spaces.
This is one example of a persistence module over a quiver of
$\A_n$-type, which in general are called \emph{zigzag persistence modules}. For more details, see \cite{zigzag}.


\subsection{Persistence Modules on Commutative Ladders}
\label{subsec:cl}
\begin{definition}
  Let $\tau=(\tau_1,\hdots,\tau_{n-1})$ be an orientation. The ladder quiver $L_n(\tau)$ is
\begin{equation*}
	L_n(\tau):
	\begin{tikzcd}
		\overset{1'}{\circ} \rar[leftrightarrow] &
    		\overset{2'}{\circ} \rar[leftrightarrow] &
		\cdots \rar[leftrightarrow] &
		\overset{n'}{\circ}
    \\
		\underset{1}{\circ} \rar[leftrightarrow] \uar[rightarrow] &
		\underset{2}{\circ} \rar[leftrightarrow] \uar[rightarrow] &
		\cdots \rar[leftrightarrow] &
		\underset{n}{\circ} \uar[rightarrow]
	\end{tikzcd}
\end{equation*}
where the directions of the arrows on both the top and bottom rows are determined by the orientation $\tau$.
The \emph{commutative ladder} $\CL_n(\tau)$ is the ladder quiver $L_n(\tau)$ bound by
commutative relations. A persistence module on the commutative ladder
$\CL_n(\tau)$ is a representation of $\CL_n(\tau)$.
\end{definition}

Recall that the \emph{\AR quiver} $\Gamma = (\Gamma_0, \Gamma_1)$ of a
bound quiver $(Q,P)$ is another quiver whose vertices $\Gamma_0$ are
given by all isomorphism classes of indecomposable representations of
$(Q,P)$, and whose arrows are given by the following. For every pair
of vertices $[M], [N] \in \Gamma_0$, $\Gamma$ has an arrow $[M]
\rightarrow [N]$ if and only if there exists an irreducible
morphism\footnote{An \emph{irreducible morphism} is a morphism
  satisfying the following two conditions: (i) $f$ is neither a
  retraction nor a section. (ii) For any factorization $f = f_1 \circ
  f_2$, either $f_1$ is a retraction or $f_2$ is a section.} $f \colon
M \rightarrow N$.

The paper \cite{PMCL} shows that for any orientation $\tau$,
$\CL_n(\tau)$ is representation-finite if and only if $n \leq 4$.
The \AR quivers of the representation-finite cases are
listed in the paper \cite{PMCL}. Figure \ref{fig:arfb} here shows the \AR
quiver of $\CL_3(fb)$.

\begin{figure}
	\centering
	\includegraphics[scale=0.3]{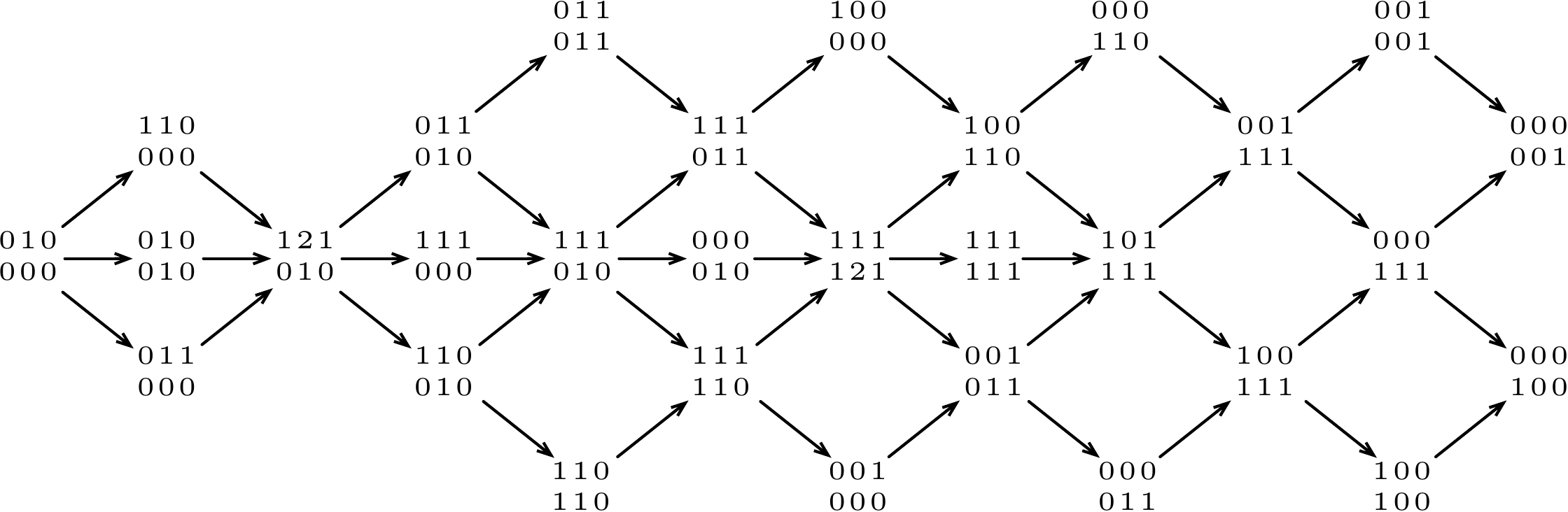}
	\caption{\AR quiver of $\CL_3(fb)$} \label{fig:arfb}
\end{figure}

The vertices of the \AR quiver in Figure~\ref{fig:arfb} are denoted by their dimension
vectors. Recall that the \emph{dimension vector} $\dim M$ of a
representation $M$ is the vector of dimensions (as $K$-vector spaces)
of $M(a)$ for vertices $a\in Q_0$. It is helpful to write the
dimension numbers $\dim_K M(a)$ corresponding to the positions
of the vertices $a\in Q_0$. For example, the dimension vector of the
indecomposable representation
\begin{equation*}
	\begin{tikzcd}
		K \rar &
    		0 \rar[leftarrow] &
		K \\
		K \rar \uar &
		K \rar[leftarrow] \uar &
		K \uar
	\end{tikzcd}
\end{equation*}
is denoted as $\dimv{101}{111}$. While the dimension vector is
invariant under isomorphism, nonisomorphic representations may have
the same dimension vector in general.

Moreover, the entries of the dimension vectors of
indecomposable representations may exceed $1$. For example, Figure~\ref{fig:arfb} has vertices  $\dimv{121}{010}$ and $\dimv{111}{121}$ representing the indecomposable representations
\begin{equation}
\label{eq:dim2_indecs}
	\begin{tikzcd}
		K \rar{\smat{1\\0}} &
    K^2 \rar[leftarrow]{\smat{0\\1}} &
		K \\
		0 \rar[rightarrow] \uar[rightarrow] &
		K \rar[leftarrow] \uar[rightarrow]{\smat{1\\1}}&
		0 \uar[rightarrow]
	\end{tikzcd}
  \text{ and }
  \begin{tikzcd}[ampersand replacement = \&]
		K \rar{1} \&
    K \rar[leftarrow]{1} \&
		K \\
		K \rar{\smat{1\\0}} \uar[rightarrow]{1} \&
		K^2 \rar[leftarrow]{\smat{0\\1}}
      \uar[rightarrow]{\smat{1&1}} \&
		K \uar[rightarrow,swap]{1}
	\end{tikzcd}
\end{equation}
respectively.

Finally we recall the definition of the \emph{persistence diagram} of
a representation $M$ of $\CL_n(\tau)$. By the above
considerations, $M$ has
\begin{equation*}
	M \cong \underset{[I]\in \Gamma_0}{\bigoplus} I^{k_{[I]}} \text{ for some }  k_{[I]} \in \mathbb{Z}_{\geq0},
\end{equation*}
where $\Gamma = (\Gamma_0, \Gamma_1)$ is the \AR quiver of
$\CL_n(\tau)$. The persistence diagram of $M$ is the map
\[
  \begin{array}{rcccl}
    D_M & \colon & \Gamma_0 & \rightarrow & \mathbb{Z}_{\geq0}\\
    & & [I] & \mapsto & k_{[I]}.
  \end{array}
\]
In the representation finite case, $\Gamma$ is a finite quiver, and we draw $D_M$ by labelling the vertices $[I]$ of $\Gamma$ with the numbers $k_{[I]}$.


\section{Main Results}
\label{sec:main_results}
We provide a decomposition algorithm for persistence modules on
commutative ladders of finite type by reinterpreting the modules as
matrices of homomorphisms between interval representations.

\subsection{From Representations to Arrows}
\label{subsec:repn_to_arrow}

\begin{definition}
  The \emph{arrow category} $\arr{\rep{Q}}$ of $\rep{Q}$ is the following category.
\begin{itemize}
\item Objects: All morphisms $\phi:V\rightarrow W$ of $\rep{Q}$, for all objects $V$ and $W$ of $\rep{Q}$.

\item Morphisms: A morphism $F = (F_V, F_W) \colon \phi_1 \rightarrow \phi_2$
  from an object $\phi_1 \colon V_1 \rightarrow W_1$ to $\phi_2 \colon
  V_2 \rightarrow W_2$ is a pair of morphisms $(F_V \colon V_1 \to
  V_2$, $F_W \colon W_1 \to W_2)$ of $\rep Q$, such that
\begin{equation*}
	\begin{tikzcd}
		V_2 \rar[rightarrow, "\phi_2"{name = U}] & W_2 \\
		V_1 \rar[rightarrow, "\phi_1"{name = D}] \uar[rightarrow, "F_V"] &
		W_1 \uar[rightarrow, swap, "F_W"]
    \arrow[phantom, to path = {(U) -- node[pos = 0.7, scale = 1.5] {$\circlearrowleft$}  (D)}]{}
	\end{tikzcd}
\end{equation*}
commutes.
\item Composition: Given $F = (F_V, F_W) \colon \phi_1 \rightarrow
  \phi_2$ and $G = (G_V, G_W) \colon \phi_2 \rightarrow \phi_3$
  \[
    G\circ F = (G_VF_V, G_WF_W)
  \]
\end{itemize}
\end{definition}
In this context, we call objects of the arrow category as arrows to
distinguish them from objects of the base category $\rep{Q}$.

\begin{theorem}
  \label{isothm}
  Let $\tau$ be an orientation. There is an isomorphism of
  $K$-categories
	\[
		\rep{\CL_n(\tau)} \cong \arr{\rep{\A_n(\tau)}}.
	\]
\end{theorem}
\begin{proof}
  An isomorphism functor
  $F \colon \rep{\CL_n(\tau)} \rightarrow \arr{\rep{\A_n(\tau)}}$
  can be constructed by taking a persistence module
  $M \in \rep{\CL_n(\tau)}$ to the morphism defined by $M$ from its
  bottom row to its top row. Similarly, a morphism between two
  persistence modules $\lambda: M\rightarrow N$ defines a morphism
  $F(\lambda)$ between the corresponding arrows $F(M)$, $F(N)$ in
  the obvious way.

\qed
\end{proof}

The isomorphism $F \colon \rep{\CL_n(\tau)} \rightarrow \arr{\rep{\A_n(\tau)}}$
constructed above allows us to identify a persistence module $M$ on $\CL_n(\tau)$
with the corresponding arrow $F(M)$.

\subsection{Arrows to a Matrix Formalism}
\label{subsec:matrixformalism}

Fix an orientation $\tau$. For ease of notation, we
define the following.
\begin{definition}
  The relation $\reltoeq$ is defined on the set of interval
  representations of $\A_n(\tau)$, $\{\intv[b,d] : 1\leq b \leq d \leq n\}$,
  by setting $\intv[a,b] \reltoeq \intv[c,d]$ if
  and only if $\Hom(\intv[a,b],\intv[c,d])$ is nonzero.
\end{definition}

It can be checked that $\reltoeq$ is reflexive and
antisymmetric\footnote{$\intv[a,b] \reltoeq \intv[c,d]$ and
  $\intv[c,d] \reltoeq \intv[a,b]$ imply $\intv[a,b] = \intv[c,d]$},
but in general is not transitive. While we use the same symbols
$\intv[b,d]$ for the intervals of any $\A_n(\tau)$, note that these
intervals and thus $\reltoeq$ depend on the underlying orientation
$\tau$. We write
$\intv[a,b] \relto \intv[c,d]$ if $\intv[a,b] \reltoeq \intv[c,d]$
and $\intv[a,b] \neq\intv[c,d]$.

\begin{lemma}
  Let $\intv[a,b],\intv[c,d]$ be interval representations of $\A_n(\tau)$.
  \label{lemma:homdim}
  \begin{enumerate}
  \item The dimension of $\Hom(\intv[a, b], \intv[c, d])$ as a $K$-vector space is
    either $0$ or $1$.
  \item A $K$-vector space basis $\left\{f\itoi{a,b,c,d}\right\}$ can be chosen
    for each nonzero $\Hom(\intv[a, b], \intv[c, d])$ such that if $\intv[a,b] \reltoeq
    \intv[c,d]$, $\intv[c,d]\reltoeq \intv[e,f]$ \emph{and}
    $\intv[a,b] \reltoeq \intv[e,f]$, then
    \begin{equation}
      f\itoi{a,b,e,f} = f\itoi{c,d,e,f} f\itoi{a,b,c,d}.
    \label{basiscondition}
    \end{equation}
  \end{enumerate}
\end{lemma}
\begin{proof}
\leavevmode
\begin{enumerate}
\item Let us use the notation $[a,b] = \{a,a+1,\hdots, b\}$ to denote the
  interval of integers $i$ with $a\leq i \leq b$ and consider $g = \{g_i\}_{i=1}^n \in \Hom(\intv[a, b], \intv[c, d])$. Suppose that $\Hom(\intv[a, b], \intv[c, d])$ is nonzero.

  Note that $g_i = 0$ for $i \notin [a,b] \cap [c,d]$. It
  follows that if $[a,b] \cap [c,d] = \emptyset$, then $\Hom(\intv[a,
  b], \intv[c, d]) = \{0\}$, a contradiction. Therefore $[a,b] \cap
  [c,d] \neq \emptyset$.

  Fix an index $j  \in [a,b] \cap [c,d] \neq \emptyset$.
  We claim that $g_i = g_j$ for any $i \in
  [a,b] \cap [c,d]$, by the commutativity requirement on morphisms. To
  see this, suppose that $i = j+1$ with $i \in [a,b] \cap [c,d]$. Then
  $g_i = g_j$ follows from the commutativity of
  \[
    \begin{tikzcd}
      K \rar[rightarrow, "\Id_K"] & K \\
      K \rar[rightarrow, "\Id_K"] \uar[rightarrow, "g_{j}"] &
      K  \uar[rightarrow, swap, "g_i"]
    \end{tikzcd}
    \;\text{ or }\;
    \begin{tikzcd}
      K \rar[leftarrow, "\Id_K"] & K \\
      K \rar[leftarrow, "\Id_K"] \uar[rightarrow, "g_{j}"] &
      K  \uar[rightarrow, swap, "g_i"]
    \end{tikzcd}
  \]
  for $\tau^i = f$ or $\tau^i = b$, respectively.
  A similar argument shows that the above claim holds
  for $i = j-1$ with $i \in [a,b] \cap [c,d]$. Repeating this
  argument, we get that $g_i = g_j$ as long as $i \in [a,b] \cap
  [c,d]$.

  Thus, any morphism $g$ is uniquely determined by its value $g_j$ for
  some $j \in [a,b]\cap [c,d]$. This provides an isomorphism of
  $K$-vector spaces
  \[
    \Hom(\intv[a,b],\intv[c,d]) \cong \Hom_K(K,K)
  \]
  by taking $g$ to $g_j$. Since $\Hom_K(K,K) \cong K$, we conclude
  that if $\Hom(\intv[a, b], \intv[c, d])$ is nonzero, then its
  dimension is $1$.

\item For all pairs of intervals with $\intv[a,b] \reltoeq
  \intv[c,d]$ define $f\itoi{a,b,c,d}$ by
 \begin{equation*}
 	\left(f\itoi{a,b,c,d}\right)_i =
 	\begin{cases}
 		\Id_K & \text{if } i \in [a,b] \cap [c,d], \\
 		0 & \text{otherwise}.
 	\end{cases}
 \end{equation*}
 The above discussion shows that $f\itoi{a,b,c,d}$ is in
 $\Hom(\intv[a, b], \intv[c, d])$, and any $g \in \Hom(\intv[a, b],
 \intv[c, d])$ can be written as $g = g_j(1) f\itoi{a,b,c,d}$ for any
 $j \in [a,b]\cap[c,d]$. Moreover, this choice of $f\itoi{a,b,c,d}$ satisfies
 Eq.~\eqref{basiscondition} by construction.
\end{enumerate}
\qed
\end{proof}

\begin{example}
With orientation $\tau =ff\cdots f$, the homomorphism spaces are
\[
  \Hom(\intv[a,b],\intv[c,d])=
  \begin{cases}
    K f\itoi{a,b,c,d}, & c \leq a \leq d \leq b, \\
    0, & \text{otherwise}.
  \end{cases}
\]
The basis functions $f\itoi{a,b,c,d}$ are given by
\[
  \left(f\itoi{a,b,c,d}\right)_{i} =
    \begin{cases}
      \Id_K, & a \leq i \leq d, \\
      0, & \text{otherwise.}
    \end{cases}
\]
With $n = 2$, $\intv[2,2] \reltoeq
\intv[1,2]$ and $\intv[1,2] \reltoeq \intv[1,1]$ but $\intv[2,2]
\not\reltoeq \intv[1,1]$. This also provides an example to illustrate
that $\reltoeq$ may not be transitive.
\end{example}

Now, let $M$ be a representation of $\CL_n(\tau)$ with $n\leq 4$. By the isomorphism
$F:\rep{\CL_n(\tau)} \cong \arr{\rep{\A_n(\tau)}}$
in Theorem~\ref{isothm}, we identify $M$ with its corresponding arrow
$F(M) \colon V \rightarrow W$
in $\arr{\rep{\A_n(\tau)}}$. Note that
$V$ is in $\rep{\A_n(\tau)}$ and thus can be decomposed as
\begin{equation}
  \label{eq:decompose_cl_v}
	\eta_V \colon V \cong
  \underset{1 \leq a \leq b \leq n}{\bigoplus} \intv[a, b]^{m_{a, b}}, \;\;
  (m_{a, b} \in \mathbb{Z}_{\geq0}\text{: multiplicity})
\end{equation}
as in Eq.~(\ref{eq:decompo_an}). A similar isomorphism $\eta_W$ can be
obtained for $W$. Through these isomorphisms, define
\begin{equation}
  \label{eq:decompose_cl_rows}
	\Phi = \eta_W F(M) \eta_V^{-1} \colon
  \underset{1 \leq a \leq b \leq n}{\bigoplus} \intv[a, b]^{m_{a, b}}
  \rightarrow
  \underset{1 \leq c \leq d \leq n}{\bigoplus} \intv[c, d]^{m'_{c, d}}.
\end{equation}
In fact, $(\eta_V,\eta_W) : F(M) \rightarrow \Phi$ is an isomorphism
in $\arr{\rep{\A_n(\tau)}}$.

Moreover, $\Phi$ can be written in a block matrix form
\begin{equation}
  \label{eq:blockmatrix}
	\Phi = \mat{\Phi\itoi{a,b,c,d}}
\end{equation}
where each block matrix entry
$\Phi\itoi{a,b,c,d} \colon \intv[a, b]^{m_{a, b}} \rightarrow \intv[c, d]^{m'_{c, d}}$
is obtained from $\Phi$ by the appropriate inclusion and projection.
That is, $\Phi\itoi{a,b,c,d}$ is the composition of
\begin{equation}
  \label{eq:block_entries_defn}
  \begin{tikzcd}
    \intv[a, b]^{m_{a, b}} \rar{\iota} &
    \underset{1 \leq a \leq b \leq n}{\bigoplus} \intv[a, b]^{m_{a, b}}  \rar{\Phi} &
    \underset{1 \leq c \leq d \leq n}{\bigoplus} \intv[c, d]^{m'_{c, d}} \rar{\pi} &
    \intv[c, d]^{m'_{c, d}}.
  \end{tikzcd}
\end{equation}
In a similar manner, each block $\Phi\itoi{a,b,c,d}$ can be
further expressed as a matrix of homomorphisms
\begin{equation*}
	\Phi\itoi{a,b,c,d} = [g_{i, j}], \;\;
	(1 \leq i \leq m_{a, b}, 1 \leq j \leq m'_{c, d}).
\end{equation*}
where each $g_{i,j} \in \Hom(\intv[a, b], \intv[c, d])$.

For intervals $\intv[a,b] \reltoeq \intv[c,d]$, part two of
Lemma~\ref{lemma:homdim} shows that for each $i$, $j$ we can write
$g_{i, j} = \mu_{i, j} f\itoi{a,b,c,d}$ for some $\mu_{i,j} \in K$.
Factoring out $f\itoi{a,b,c,d}$ from $\Phi\itoi{a,b,c,d}$ with $\intv[a,b] \reltoeq \intv[c,d]$, we get
\begin{equation*}
  \Phi\itoi{a,b,c,d} =
  \begin{cases}
    M\itoi{a,b,c,d} f\itoi{a,b,c,d} & \text{if } \intv[a,b] \reltoeq \intv[c,d], \\
    0 & \text{otherwise},
  \end{cases}
\end{equation*}
where each $M\itoi{a,b,c,d}$ is an $m_{c,d}' \times m_{a,b}$ matrix with entries in $K$.
To summarize, we define the following.
\begin{definition}
  Let $M$ be a persistence module on $\CL_n(\tau)$. The block matrix form $\Phi(M)$  of $M$ is
\begin{equation}
\label{eq:blockmatrix_factored}
	\Phi(M) = \mat{\Phi\itoi{a,b,c,d}} = \mat{M\itoi{a,b,c,d} f\itoi{a,b,c,d}}_{\intv[a,b] \reltoeq \intv[c,d]}.
\end{equation}
where each $\Phi\itoi{a,b,c,d}$ is as defined in Eq.~\eqref{eq:block_entries_defn}.
\end{definition}
In the matrix formalism, we label the rows and columns of the block
matrix corresponding to the summand $\intv[a,b]^{m_{ab}}$ by
$\Int{a,b}$. We say that the block $\Phi\itoi{a,b,c,d}$ is in
row $\Int{c,d}$ and column $\Int{a,b}$.

\subsection{Permissible operations}
\label{subsec:permissibles}
While we have written $\Phi$ in a block
matrix form, not all of the usual row and column operations on $K$-matrices
correspond to a meaningful change of basis.
The fact that there exist some pairs of intervals where
$\Hom(\intv[a, b], \intv[c, d])$ is zero leads to some complications.

If $(R,S): \Phi' \cong \Phi$ is an isomorphism, then
\[
\begin{tikzcd}
\bigoplus\limits_{1\leq a \leq b \leq n} \intv[a,b]^{m_{a,b}} \rar{\Phi}  &
 \bigoplus\limits_{1\leq a \leq b \leq n} \intv[a,b]^{m'_{a,b}} \\
\bigoplus\limits_{1\leq a \leq b \leq n} \intv[a,b]^{m_{a,b}} \rar{\Phi'} \uar{R} &
 \bigoplus\limits_{1\leq a \leq b \leq n} \intv[a,b]^{m'_{a,b}} \uar{S}
\end{tikzcd}
\]
commutes and $\Phi' = S^{-1} \Phi R$.
Observing that the domain and codomain of $R$ are direct summations,
$R$ can be written in a matrix form $R = \left[R\itoi{a,b,c,d} f\itoi{a,b,c,d}\right]_{\intv[a,b] \reltoeq \intv[c,d]}$ relative to them, by an argument similar to that done for $\Phi$. Similarly, $S$ can be written in a matrix form.
It can be checked that $(R,S):\Phi' \rightarrow \Phi$ is an isomorphism if and only if all the diagonal entries $R\itoi{a,b,a,b}$ and $S\itoi{a,b,a,b}$ are invertible.

Let us discuss column
operations and assume $S$ is the identity.
In analogy to usual linear algebra, column operations on $\Phi$ correspond to
a change of interval summands induced by $R$. To see this, let us
choose a column $\Int{a,b}$ and suppose that $\intv[a,b] \reltoeq
\intv[c,d]$.

The block entry at row $\Int{c,d}$ in column $\Int{a,b}$ of $\Phi' =
\Phi R$ is
\[
  \begin{array}{rcl}
    [\Phi R]\itoi{a,b,c,d}
                  & = & \displaystyle\sum\limits_{\intv[a,b] \reltoeq \intv[e,f] \reltoeq \intv[c,d]}
                        (M\itoi{e,f,c,d} f\itoi{e,f,c,d})
                        (R\itoi{a,b,e,f} f\itoi{a,b,e,f})
    \\
                  & = & \left(\displaystyle\sum_{\intv[a,b] \reltoeq \intv[e,f] \reltoeq \intv[c,d]}
                        M\itoi{e,f,c,d} R\itoi{a,b,e,f} \right)f\itoi{a,b,c,d}
  \end{array}
\]
where $\Phi R$ is computed as a usual multiplication of block
matrices. In the last step, we used the property that $f\itoi{e,f,c,d}
f\itoi{a,b,e,f} = f\itoi{a,b,c,d}$ as guaranteed by
Lemma~\ref{lemma:homdim}. Note that the resulting coefficient of
$f\itoi{a,b,c,d}$ above involves only addition and multiplication of
$K$-matrices. Furthermore, since $\intv[a,b] \reltoeq \intv[a,b]$, it is equal to
\[\displaystyle\sum_{\intv[a,b] \reltoeq \intv[e,f] \reltoeq \intv[c,d]}
  M\itoi{e,f,c,d} R\itoi{a,b,e,f}
  =
  \left(M\itoi{a,b,c,d} R\itoi{a,b,a,b} +
    \displaystyle\sum_{\intv[a,b] \relto \intv[e,f] \reltoeq \intv[c,d]}
    M\itoi{e,f,c,d} R\itoi{a,b,e,f} \right).
\]
In this form, we see that apart from a change of basis
$R\itoi{a,b,a,b}$ within the column $\Int{a,b}$, we also permit
addition of multiples of columns $\Int{e,f}$ with $\intv[a,b] \relto
\intv[e,f]$. A similar analysis can be performed for row operations.

For example, let us consider a persistence module $M$ on $\CL_2(f)$ corresponding to
\begin{equation}
\Phi=
\begin{tikzpicture}[baseline=(p.center)]
  \matrix[matrix of math nodes, row sep=0em, column sep=0em,
  ampersand replacement = \&,left delimiter=[,right delimiter={]}](p){
    M\itoi{2,2,2,2}\Id_{\intv[2,2]}
    \& 0
    \& 0
    \\
    M\itoi{2,2,1,2}f\itoi{2,2,1,2}
    \& M\itoi{1,2,1,2} \Id_{\intv[1,2]}
    \&  0
    \\
    0
    \& M\itoi{1,2,1,1}f\itoi{1,2,1,1}
    \& M\itoi{1,1,1,1} \Id_{\intv[1,1]}
    \\
  };
  \node[anchor=south east, left=6pt] (p-0-0) at (p-1-1.north west) {};

  \foreach[count=\i] \v in {\(\Int{2,2}\),\(\Int{1,2}\), \(\Int{1,1}\)}{
    \node (p-\i-0) at (p-0-0 |- p-\i-\i) {};
    \path (p-\i-0.north) -- (p-\i-0.south) node [midway, left] { \v };
  }
  \foreach[count=\i] \v in {\(\Int{2,2}\),\(\Int{1,2}\),\(\Int{1,1}\)}{
    \node (p-0-\i) at (p-\i-\i |- p-0-0) {};
    \path (p-0-\i.west) -- (p-0-\i.east) node [midway, above] { \v };
  }
\end{tikzpicture}.
\label{eq:generall2l2}
\end{equation}
Since $\Int{1,1} \not\reltoeq
\Int{1,2}$, $\Int{1,1} \not\reltoeq \Int{2,2}$, $\Int{1,2}
\not\reltoeq \Int{2,2}$, and $\Int{2,2} \not\reltoeq \Int{1,1}$, the blocks in the corresponding positions are
zero. An automorphism $R$ on $V = \intv[2,2]^{m_{2,2}} \oplus \intv[1,2]^{m_{1,2}} \oplus \intv[1,1]^{m_{1,1}}$ as defined above can be written as
\[
R =
\begin{tikzpicture}[baseline=(p.center)]
  \matrix[matrix of math nodes, row sep=0em, column sep=0em,
  ampersand replacement = \&,left delimiter=[,right delimiter={]}](p){
    R\itoi{2,2,2,2}\Id_{\intv[2,2]}
    \& 0
    \& 0
    \\
    R\itoi{2,2,1,2} f\itoi{2,2,1,2}
    \& R\itoi{1,2,1,2} \Id_{\intv[1,2]}
    \&  0
    \\
    0
    \& R\itoi{1,2,1,1} f\itoi{1,2,1,1}
    \& R\itoi{1,1,1,1} \Id_{\intv[1,1]} \\
  };
  \node[anchor=south east, left=6pt] (p-0-0) at (p-1-1.north west) {};

  \foreach[count=\i] \v in {\(\Int{2,2}\),\(\Int{1,2}\), \(\Int{1,1}\)}{
    \node (p-\i-0) at (p-0-0 |- p-\i-\i) {};
    \path (p-\i-0.north) -- (p-\i-0.south) node [midway, left] { \v };
  }
  \foreach[count=\i] \v in {\(\Int{2,2}\),\(\Int{1,2}\),\(\Int{1,1}\)}{
    \node (p-0-\i) at (p-\i-\i |- p-0-0) {};
    \path (p-0-\i.west) -- (p-0-\i.east) node [midway, above] { \v };
  }
\end{tikzpicture}.
\]
Thus, $\Phi R$ is
\[
\begin{tikzpicture}[baseline=(p.center)]
  \matrix[matrix of math nodes, row sep=0em, column sep=0em,
  ampersand replacement = \&,left delimiter=[,right delimiter={]}](p){
    M\itoi{2,2,2,2} R\itoi{2,2,2,2}\Id_{\intv[2,2]}
  \& 0
  \& 0
  \\
  (M\itoi{2,2,1,2} R\itoi{2,2,2,2} + M\itoi{1,2,1,2} R\itoi{2,2,1,2}) f\itoi{2,2,1,2}
  \&  M\itoi{1,2,1,2} R\itoi{1,2,1,2}\Id_{\intv[1,2]}
  \& 0
  \\
  (M\itoi{1,2,1,1} R\itoi{2,2,1,2}) (f\itoi{1,2,1,1} f\itoi{2,2,1,2})
  \& (M\itoi{1,2,1,1} R\itoi{1,2,1,2} + M\itoi{1,1,1,1} R\itoi{1,2,1,1}) f\itoi{1,2,1,1}
  \&  M\itoi{1,1,1,1} R\itoi{1,1,1,1}\Id_{\intv[1,1]}
  \\
  };
\end{tikzpicture}.
\]
Since $f\itoi{1,2,1,1} f\itoi{2,2,1,2} = 0$, the lower left corner is still a zero block, as can be expected.


\section{Algorithm}
\label{sec:algorithms}

\subsection{Input and Notation}
\label{subsec:algo_notations}
Let $n\leq 4$, $\tau$ be an orientation, and $M$ be a persistence module on $\CL_n(\tau)$. In the previous section, we constructed the block matrix $\Phi(M) = \mat{\Phi\itoi{a,b,c,d}} = \mat{M\itoi{a,b,c,d}f\itoi{a,b,c,d}}$ associated to $M$.
Two blocks at distinct entries are said to be column (row) neighbors if
they are in the same column (row). Neighbors do not have
to be directly adjacent to each other in the block matrix.
By an abuse of notation, we let $(\Int{c,d},\Int{a,b})$
referring to row $\Int{c,d}$ and column $\Int{a,b}$ also refer
to the block (submatrix) located at that entry.

Recall that the vertices of the \AR quiver $\Gamma(\A_n(\tau))$ of $\A_n(\tau)$ are in bijective correspondence to the interval representations.
Hence, the quiver structure of $\Gamma(A_n(\tau))$ naturally induces a partial order on
the set of intervals, by going from source vertices to sink vertices. We fix a total order $\prec$ extending this by resolving ambiguities using reverse lexicographic order on the pairs $(b,d)$. For example, since $\A_n(fb)$ has \AR quiver
\[
\begin{tikzpicture}[baseline=(A.center)]
    \matrix (A) [matrix of math nodes, column sep= 4 mm, row sep = 2
    mm, nodes in empty cells]
    {
                &  \Int{1,2} &           & \Int{3,3}    \\
      \Int{2,2} &            & \Int{1,3} &  \\
                &  \Int{2,3} &           & \Int{1,1}   \\
    };

    \path (A-2-3)
    edge [->] (A-1-4)
    edge [->] (A-3-4)
    edge [<-] (A-1-2)
    edge [<-] (A-3-2);

    \path [->] (A-2-1) edge (A-1-2) edge (A-3-2);
  \end{tikzpicture}
\]
we get the order $\Int{2,2} \prec \Int{2,3} \prec \Int{1,2} \prec \Int{1,3} \prec \Int{3,3} \prec \Int{1,1}$. Here, the ambiguities are resolved as $\Int{2,3} \prec \Int{1,2}$, and $\Int{3,3} \prec \Int{1,1}$. We shall use $\prec$ to order the columns and rows of the block matrix.

Finally, we define the data that serves as input to Algorithm~\ref{algo:main}.
\begin{definition}
  \label{defn:algo_input}
  Let $M$ be a persistence module on $\CL_n(\tau)$. The \emph{block matrix
  problem} of $M$ is the block matrix $\Phi(M)$, together with permissible operations, and with rows and columns ordered as below.
  \begin{enumerate}
  \item If $\intv[a,b] \reltoeq \intv[c,d]$ then operations from
    row $\Int{a,b}$ to row $\Int{c,d}$ are permissible.
  \item If $\intv[a,b] \reltoeq \intv[c,d]$ then operations to
    column $\Int{a,b}$ from column $\Int{c,d}$ are permissible.
  \end{enumerate}
  The columns from left to right are ordered in increasing $\prec$ while rows from top to bottom are in decreasing $\prec$.
\end{definition}

Note that the permissible operations are the rules derived in Subsection~\ref{subsec:permissibles} and that applying these operations result in a block matrix isomorphic to $\Phi(M)$. For convenience, we distinguish the permissible operations that operate only within a fixed row or column block. An \emph{inner row (column) operation} is a
row (column) operation that only affects
$K$-vector rows (columns) within some fixed row (column)
$\Int{a,b}$.

Recall that if $\intv[a,b] \not\reltoeq \intv[c,d]$, then the block
$(\Int{c,d}, \Int{a,b})$ is always zero, even after the application of
any permissible operations. To distinguish them from other blocks that
just happen to be numerically zero, we denote them by $\n$ and call these blocks
\emph{strongly zero blocks}.

Otherwise, for $\intv[a,b] \reltoeq \intv[c,d]$, we use the symbol
$*$ to abbreviate the block $M_{\Int{a, b}}^{\Int{c, d}}$ at $(\Int{c,d}, \Int{a,b})$. This indicates that these blocks are so far \emph{unprocessed}. As we operate on the block
matrix, their status as unprocessed will be changed to either an
identity matrix $E$ or a zero matrix $0$.

\begin{notation}
  To denote the possible block statuses, we use:
  \begin{itemize}
  \item $*$ for unprocessed blocks,
  \item $\n$ for strongly zero blocks,
  \item $E$ for identity blocks (of appropriate sizes),
    and
  \item $0$ for zero blocks (of appropriate sizes).
  \end{itemize}    
  The blocks marked as $\n$, $E$, and $0$ are considered processed.
\end{notation}

Note that the block matrix may have numerically identity or zero blocks, even though we label their status as being unprocessed $*$. This status only reflects the fact that they have not yet been examined and fixed through the course of the algorithm.

\begin{example}
  \label{example:clfb_mp}
  The block matrix problem corresponding to a persistence module on $\CL_3(fb)$ has the form
  \begin{equation}
    \mattikz{
      \matheader{
        \n \& \n \& \* \& \* \& \n \& \* \\
        \n \& \* \& \n \& \* \& \* \& \n \\
        \* \& \* \& \* \& \* \& \n \& \n \\
        \* \& \n \& \* \& \n \& \n \& \n \\
        \* \& \* \& \n \& \n \& \n \& \n \\
        \* \& \n \& \n \& \n \& \n \& \n \\
      };
      \node[anchor = south east] (p-0-0) at (p-2-1.west |- p-1-1.north west) {};
      \foreach[count = \i] \v in {
        \(\Int{1, 1}\), \(\Int{3, 3}\),
        \(\Int{1, 3}\), \(\Int{1, 2}\),
        \(\Int{2, 3}\), \(\Int{2, 2}\)}
      {
        \node[left=3.5ex] at (p-\i-1) {\scriptsize\v};
      }
      \foreach[count = \i] \v in {
        \(\Int{2, 2}\), \(\Int{2, 3}\),
        \(\Int{1, 2}\), \(\Int{1, 3}\),
        \(\Int{3, 3}\), \(\Int{1, 1}\)}
      {
        \node[above=1.5ex] at (p-1-\i) {\scriptsize\v};
      }
    }.
  \end{equation}
\end{example}


\subsection{Algorithm}
\label{subsec:pseudocode}
Given a persistence module $M$ on $\CL_n(\tau)$ where $n\leq 4$, the
input to Algorithm~\ref{algo:main} is the block matrix problem of $M$.
Below, we shall also use the notation $M$ to denote the block matrix
problem associated to the persistence module $M$.

The algorithm uses the following two facts.
Given a usual $K$-matrix $N$, there exist invertible matrices $R$ and
$S$ (of appropriate sizes) such that $RNS = \left[\SNF\right]$, a
Smith normal form. Thus, by using appropriate inner row and column operations,
a block $*$ can be transformed into the form $\left[\SNF\right]$.
Meanwhile, using some identity submatrix $E$, a row (column) neighbor
$*$ can be zeroed out using appropriate permissible column (row) operations. Complications come from the side effects of these operations.

\begin{algorithm}[H]
	\caption{Main Algorithm. Input: a block matrix problem $M$}
  \label{algo:main}
	\begin{algorithmic}[1]
		\Procedure{matrix\_reduction}{$M$}
			\While{$M$ has unprocessed submatrices}
				\State $v_* \gets$ the bottommost $*$ block of the rightmost column with $*$ blocks in $M$ \label{line:main_choice}
				\State $F_R \gets$ \textproc{row\_side\_effect}($v_*$)
				\State $F_C \gets$ \textproc{col\_side\_effect}($v_*$)
				\State Transform $v_*$ to Smith normal form by inner operations on $M$. \label{line:main_snf}
        \ForDo{$v' \in F_R$}{\textproc{col\_fix}($v'$)} \label{line:main_colfix}
				\ForDo{$v' \in F_C$}{\textproc{row\_fix}($v'$)} \label{line:main_rowfix}
				\State Update the partitioning of blocks in block matrix $M$. \label{line:main_part}
        \While{there exist nonzero blocks $v_t$ with $(p \gets \textproc{erasable}(v_t))$ not null} \label{line:main_erasable}
          \State Zero out $v_t$ via the procedure indicated by $p$.
        \EndWhile
      \EndWhile
		\EndProcedure
	\end{algorithmic}
\end{algorithm}

The main while loop of Algorithm~\ref{algo:main} can be
divided broadly into four main parts.
\begin{enumerate}
\item Transform one appropriate block $v_*$ into a Smith normal form
  (line \ref{line:main_snf}) by inner row and column operations on $M$.
\item The operations performed in the previous part may affect the
  forms of neighboring identity blocks. We transform them back to 
  identity blocks (lines \ref{line:main_colfix} and \ref{line:main_rowfix}).
\item Update the partitioning of the blocks (line \ref{line:main_part}). After obtaining
  the Smith normal form $\left[\SNF\right]$, we split up columns and
  rows so that each identity matrix $E$ is its own block in $M$.
\item Greedily zero out erasable blocks $v$ by addition of multiples of
  identity blocks.
\end{enumerate}

We first illustrate parts one to three by an example. Suppose that
\[
  M = 
  \mattikz{
    \matheader{
      * \& * \& v_* \\
      * \& \n \& E  \\     
    };
    \node[right=6pt] (LR) at (p-1-3.east |- p-2-3.south east) {};
    \rowarrow{[<->](LR |- p-2-3.east)}{(LR |- p-1-3.east)};
  }
\]
where operations from row $1$ to row $2$ and vice versa are impermissible.
We get:
\[
  M = 
  \mattikz{
    \matrix[matrix of math nodes, column sep=0mm, row sep=0mm,
    inner sep = 0mm, left delimiter = {[},right delimiter = {]},
    every node/.append style={
      anchor=center,text depth = 0.5em,text
      height=1em,minimum width=1.5em}](p){
      * \& * \& v_* \\
      * \& \n \& E  \\
    };
  }
  \overset{1.}{\cong}
  \mattikz{
    \matrix[matrix of math nodes, column sep=0mm, row sep=0mm,
    inner sep = 0mm, left delimiter = {[},right delimiter = {]},
    every node/.append style={
      anchor=center,text depth = 0.5em,text
      height=1em,minimum width=1.5em}](p){
      * \& * \& \SNF \\
      * \& \n \& S  \\
    };
  }
  \overset{2.}{\cong}
  \mattikz{
    \matrix[matrix of math nodes, column sep=0mm, row sep=0mm,
    inner sep = 0mm, left delimiter = {[},right delimiter = {]},
    every node/.append style={
      anchor=center,text depth = 0.5em,text
      height=1em,minimum width=1.5em}](p){
      * \& * \& \SNF \\
      * \& \n \& E  \\
    };
  },
\]
where the numbers above the isomorphisms indicate the procedures being
performed.

In the first part, the block $v_*$ is chosen by the heuristic given in
Algorithm~\ref{algo:main}, line \ref{line:main_choice}. Note that $v_*$ is therefore
dependent on the ordering of the rows and columns, which we have fixed
in Definition~\ref{defn:algo_input}.
By inner operations on $M$, the block $v_*$ is transformed to Smith normal
form. In particular, there are invertible matrices $R$ and $S$ such
that $Rv_*S = \left[ \SNF \right]$.

\begin{algorithm}
\caption{}
	\begin{algorithmic}[1]
		\Function{col\_side\_effect}{$v$}
    \State \Return $\set{v' | v' \text{ is an identity column neighbor of } v}$
		\EndFunction
	\end{algorithmic}
  \begin{algorithmic}[1]
		\Function{row\_side\_effect}{$v$}
			\State \Return $\set{v' | v' \text{ is an identity row neighbor of } v}$
		\EndFunction
	\end{algorithmic}
	\label{algo:sideeffects}
\end{algorithm}

Next, the block $E$ below $v_*$ becomes $E S = S$, possibly
not an identity matrix. This is recorded as a side effect. Since $S$ is invertible, it can be transformed back by using only inner row operations in \textproc{row\_fix}.
In general there may be other
identity blocks in the same row as $S$ whose forms are affected by
these row operations. To fix them, we recursively call \textproc{row\_fix} and \textproc{col\_fix} in Algorithm~\ref{algo:colrowfix}. Checking that this does not lead to an infinite
recursion for the cases we consider is part of the proof of Theorem~\ref{main_theorem}.

\begin{algorithm}
	\caption{}
	\begin{algorithmic}[1]
		\Function{col\_fix}{$v$}
			\State $V' \gets \textproc{col\_side\_effect}(v)$
			\State Transform $v$ to an identity by inner column operations on $M$.
      \ForDo{$v' \in V'$}{$\textproc{row\_fix}(v')$}
		\EndFunction
	\end{algorithmic}
  \begin{algorithmic}[1]
		\Function{row\_fix}{$v$}
			\State $V' \gets \textproc{row\_side\_effect}(v)$
			\State Transform $v$ to an identity by inner row operations on $M$.
      \ForDo{$v' \in V'$}{$\textproc{col\_fix}(v')$}
		\EndFunction
	\end{algorithmic}
	\label{algo:colrowfix}
\end{algorithm}

Next is part three, where we update the block matrix partitioning to
isolate the identity blocks $E$. Both the row and column of $v_*$ are split into two. We get
\[
  M \cong \hdots \overset{3.}{=}
  \mattikz{
    \matheader{
      * \& * \& E \& 0 \\
      * \& * \& 0 \& 0 \\
      * \& \n \& E \& 0 \\
      * \& \n \& 0 \& E \\
    };
  }
\]
Since $v_*$ has a column neighbor $E$, the bottom row
also needs to be split to isolate the parts of the old identity block.

Finally, we discuss part four. A simple case for a target block $v_t$ to be erasable is when $v_t = (r,c)$ has a column neighbor identity block $v_E = (r',c)$ that has no nonzero row neighbors, and such that row operations from
row $r'$ to row $r$ are permissible. Using permissible row operations,
the block $v_t$ can be zeroed out by addition of a multiple of the identity block $v_E$. A similar statement holds if there exists a row neighbor identity block $v_E$ satisfying similar conditions.

The above cases present no side effects. In general, zeroing out the target block $v_t$ by addition of multiples of a
row (column) may change the forms of other processed blocks. We separate the cases of row and column erasability in Algorithm~{\ref{alg:erasable}}.

\begin{algorithm}
	\caption{}
	\begin{algorithmic}[1]
		\Function{erasable}{$v_t$, $v_f$ = null, $visited = \{ \}$}
			\If{\textproc{row\_erasable}($v_t$, $v_f$, $visited$) is not null}
				\State \Return \textproc{row\_erasable}($v_t$, $v_f$, $visited$)
        		\ElsIf{\textproc{col\_erasable}($v_t$, $v_f$, $visited$) is not null}
				\State \Return \textproc{col\_erasable}($v_t$, $v_f$, $visited$)
			\Else
				\State \Return null
			\EndIf
		\EndFunction
	\end{algorithmic}
	\label{alg:erasable}
\end{algorithm}

In zeroing out the target $v_t$, we avoid changing the forms of any previously obtained identity blocks. It is also possible that a zero block $v_t'$
may become nonzero as a side effect. The
algorithm ensures that if this happens, then $v_t'$ can and will
be transformed back to $0$ again. Iteratively, repairing these side
effects may introduce more side effects. Thus, we recursively call on
our check for erasability on each side effect. To avoid any infinite recursion, we keep track of the targets $v_t$ visited, and visit each block as a target at most once for each top-level call to \textproc{erasable}.

If the above conditions can be satisfied, the function \textproc{erasable} returns a finite directed tree, called the \emph{process tree}, that records the procedure to zero
out $v_t$. Each vertex in a process tree is labelled with a pair $(v_t,
v_E)$ of a target block and an identity block that can be used to zero
out $v_t$. The successor vertices $(v_t', v_E')$ of a vertex $(v_t,v_E)$ consist of all $v_t'$ that appear as side effects in the operation to zero out $v_t$ using $v_E$.

If no such procedure can be found, then \textproc{erasable} returns a null (empty) process tree.
This means that the block in question is declared as not being
erasable in the current step of the algorithm.

\begin{algorithm}
	\caption{Check whether or not $v_t$ is row erasable without using
    block $v_f$.}
	\begin{algorithmic}[1]
		\Function{row\_erasable}{$v_t$, $v_f$, ${visited}$}
    \State $V'$ = \textproc{col\_side\_effect}($v_t$) \label{line:sideeffect}
    \State $visited \gets visited \cup \set{v_t}$
    \ForAll{$v_E \in V'$ not in $visited$, $v_E \neq v_f$, and row operation from $v_E$ to $v_t$ permissible}
      \State $usable \gets$ true; $subtrees \gets \{\}$ 
      \ForAll{$u \in$ \textproc{nonzero\_row\_neighbors}($v_E$)}
	      \State $v_t' \gets$ the block in same row as $v_t$ and same column as $u$. \label{line:newtarget}
        \If{$v_t'$ is in $visited$ or $v_t' = E$}
          \State $usable \gets$ false; \textbf{break}
        \EndIf        
        \If{($v_t' = 0$ and $(p \gets \textproc{erasable}(v_t', u, visited)) = $ null)}\label{line:erasable_recurse}
          \State $usable \gets$ false; \textbf{break}
        \EndIf
        \State $subtrees \gets  subtrees \cup \{p\}$.
      \EndFor
      \If{$usable$}
        \State \Return process tree with root $(v_t, v_E)$ and arrows to the roots of $subtrees$.
        \EndIf
    \EndFor
    \State \Return null 
		\EndFunction
	\end{algorithmic}
	\label{algo:rowerasable}
\end{algorithm}

Let us discuss \textproc{row\_erasable} in Algorithm~\ref{algo:rowerasable} in detail.
In line~\ref{line:sideeffect}, we use the function
$\textproc{col\_side\_effect}(v_t)$  to get candidate identity
blocks $v_E$. We consider only unvisited blocks $v_E$ where
the row operation from $v_E$ to $v_t$ is permissible, and where $v_E$ is not the flagged block $v_f$.
Its purpose will become clear below.


Now, $\textproc{nonzero\_row\_neighbors}(v_E)$ is defined
to return the set of row neighbors $u$ of $v_E$ that are not zero nor
strongly zero. Each $u$ can potentially induce a side effect, which we check one by one.
To illustrate, consider the following
arrangement
\[
  \mattikz{
    \matheader{
      \& \vdots \& \& \vdots \&  \\
      \cdots \& u \& \cdots \& v_E \&  \cdots\\
      \& \vdots \& \& \vdots \&  \\
      \cdots \& v_t' \& \cdots \& v_t \&  \cdots\\
      \& \vdots \& \& \vdots \&  \\
    };
    \node[anchor = south east, left = 4pt] (p-0-0) at (p-2-1.west |- p-1-2.north west) {};

    \node (p-2-0) at (p-0-0 |- p-2-2) {};
    \node (p-4-0) at (p-0-0 |- p-4-4) {};
    \path (p-2-0.north) -- (p-2-0.south) node [midway, left] { $r_1$ };
    \path (p-4-0.north) -- (p-4-0.south) node [midway, left] { $r_2$ }; 
  }
\]
where $v_E$ is the identity block under consideration.
Here, $u$ is a nonzero row neighbor of $v_E$. Since we want to add multiples of row $r_1$ to
$r_2$ to zero out $v_t$, the block $v_t'$ in same row $r_2$ as $v_t$
and same column as $u$ (Line~\ref{line:newtarget}) may possibly have
its form affected.

The next few lines handle the checking of block $v_t'$. If the block
$v_t'$ is an identity block, or if it has already been visited
previously, then we do not use row $r_1$. If the block $v_t'$ is zero, we need to check whether or not it can be transformed back to zero again.
Here, the flag $v_f$ comes into play. We set the
flagged block as $v_f = u$ in the call to \textproc{erasable} in Line~\ref{line:erasable_recurse}, since
we do not want to use $u$ to zero out $v_t'$, thereby undoing the
operations to zero out $v_t$.

\begin{algorithm}
  \caption{Check whether or not $v_t$ is column erasable without using
    block $v_f$.}
	\begin{algorithmic}[1]
		\Function{col\_erasable}{$v_t$, $v_f$, $visited$}
    \State $V'$ = \textproc{row\_side\_effect}($v_t$)
    \State $visited \gets visited \cup \set{v_t}$
    \ForAll{$v_E \in V'$ not in $visited$, $v_E \neq v_f$, and column operation from $v_E$ to $v_t$ permissible}
      \State $usable \gets$ true; $subtrees \gets \{\}$
      \ForAll{$u \in$ \textproc{nonzero\_col\_neighbors}($v_E$)}
        \State $v_t' \gets$ the block in same column as $v_t$ and same row as $u$. \label{line:colnewtarget}
        \If{ $v_t'$ is in $visited$ or $v_t' = E$}
          \State $usable \gets$ false; \textbf{break}
        \EndIf
        \If{($v_t' = 0$ and $(p \gets \textproc{erasable}(v_t', u, visited)) = $ null)}
          \State $usable \gets$ false; \textbf{break}
        \EndIf
        \State $subtrees \gets  subtrees \cup \{p\}$.
      \EndFor
      \If{$usable$}
    \State \Return process tree with root $(v_t, v_E)$ and arrows to the roots of $subtrees$.
    \EndIf
    \EndFor
    \State \Return null
		\EndFunction
	\end{algorithmic}
	\label{algo:colerasable}
\end{algorithm}

If a nonempty process tree is returned by the top-level call to
$\textproc{erasable}(v_t)$ in Algorithm~\ref{algo:main}, then $v_t$ is erasable. By construction, it suffices to traverse the process tree and do the
operations indicated to zero out $v_t$ and fix all side effects.

Let us reproduce here Theorem~\ref{main_theorem} concerning Algorithm~\ref{algo:main}.

\noindent\parbox{\textwidth}{\algomainthm*}

Whether or not Algorithm~\ref{algo:main} terminates depends not on the particular persistence module, but on the statuses of the blocks and the status changes brought about by the operations. Moreover, the operations to be performed only depends on the arrangement of the statuses. All these depend only on the initial arrangement, which in turn depends on the orientation $\tau$ and the ordering chosen for the intervals.

From a result in \cite{PMCL}, a
commutative ladder $\CL_n(\tau)$ is finite type if and only if $n \leq
4$, so that there are only a finite number of cases to check.
Below, we provide proofs for Theorem~\ref{main_theorem} with orientations $f$, $fb$, and
$fff$.  The proofs for the other orientations are similar.

Furthermore, an indecomposable decomposition can easily be read off
the resulting normal form consisting of only identity, zero, and
strongly zero block, and the correspondence to an indecomposable
decomposition of the persistence module $M$ is provided by Theorem~\ref{isothm}.

We were unable to find a proof that does not involve manually checking each possible orientation. Given a particular persistence module, it is clear for each completed iteration of the main while loop in Algorithm~\ref{algo:main}, the total number of scalar entries in unprocessed blocks strictly decreases. Moreover, the procedure \textproc{erasable} avoids any infinite recursion by construction. The difficulty comes from the use of Algorithm~\ref{algo:main}, line~\ref{line:main_choice} for choosing $v_*$ and subsequently showing that all side effects can always be resolved.


\subsubsection{Case $\CL_2(f)$}
\label{subsubsec:clf}
The input block matrix problem is of the form
\begin{equation*}
	\mattikz{
    \matheader{
      \n \& \* \& \* \\
      \* \& \* \& \n \\ 
      \* \& \n \& \n \\
		};
		\foreach[count = \i] \v in {\(\Int{1, 1}\), \(\Int{1, 2}\), \(\Int{2, 2}\)}{
      \node[left=2.5ex] at (p-\i-1) {\scriptsize\v};
		}
		\foreach[count = \i] \v in {\(\Int{2, 2}\), \(\Int{1, 2}\), \(\Int{1, 1}\)}{
      \node[above=1.5ex] at (p-1-\i) {\scriptsize\v};
		}
		\node[right=6pt] (LR) at (p-1-3.east |- p-3-3.south east) {};
		\rowarrow{(LR |- p-3-3.east)}{(LR |- p-1-3.east)};
		\colarrow{(p-3-3.south |- LR)}{(p-3-1.south |- LR)};
	}
\end{equation*}
in general. Initially, all top to bottom and left to right operations
are impermissible. The red arrows show the additional impermissible
operations.

First the unprocessed block at $v_* = (\Int{1,1}, \Int{1,1})$
is transformed by inner elementary operations to Smith normal form
$\left[\SNF\right]$. Note that $v_*$ has no identity neighbors so that
there are no side effects to undo.

Updating the block partitioning, the matrix is now in the form
\begin{equation*}
	\mattikz{
    \matheader{
      \n \& \* \& \E \& \0 \\
      \n \& \* \& \0 \& \0 \\
      \* \& \* \& \n \& \n \\ 
      \* \& \n \& \n \& \n \\
		};
		\foreach[count = \i] \v in {\(\Int{1, 1}_1\), \(\Int{1, 1}_2\), \(\Int{1, 2}\), \(\Int{2, 2}\)}{
      \node[left=2.5ex] at (p-\i-1) {\scriptsize\v};
		}
		\foreach[count = \i] \v in {\(\Int{2, 2}\), \(\Int{1, 2}\), \(\Int{1, 1}_1\), \(\Int{1, 1_2}\)}{
      \node[above=1.5ex] at (p-1-\i) {\scriptsize\v};
		}
		\node[right=6pt] (LR) at (p-1-4.east |- p-4-4.south east) {};
		\colarrow{(p-3-3.south |- LR)}{(p-3-1.south |- LR)};
	}.
\end{equation*}
For convenience, we use subscripts to distinguish the two columns and rows
corresponding to $\Int{1,1}$ obtained after the repartitioning.
Additions from the columns in $\Int{1,1}_1$ to the columns in
$\Int{1,2}$ are permitted, and the unprocessed submatrix $(\Int{1,1}_1,\Int{1,2})$ is
erasable using the newly processed $E$, without any side effects. We get the form
\begin{equation} \label{CLf1}
	\mattikz{
    \matheader{
      \n \& \0 \& \E \& \0 \\
      \n \& \* \& \0 \& \0 \\
      \* \& \* \& \n \& \n \\ 
      \* \& \n \& \n \& \n \\
		};
		\foreach[count = \i] \v in {\(\Int{1, 1}_1\), \(\Int{1, 1}_2\), \(\Int{1, 2}\), \(\Int{2, 2}\)}{
      \node[left=2.5ex] at (p-\i-1) {\scriptsize\v};
		}
		\foreach[count = \i] \v in {\(\Int{2, 2}\), \(\Int{1, 2}\), \(\Int{1, 1}_1\), \(\Int{1, 1_2}\)}{
      \node[above=1.5ex] at (p-1-\i) {\scriptsize\v};
		}
	}
  \cong
  \mattikz{
    \matheader{
      \n \& \*  \\
      \* \& \*   \\ 
      \* \& \n  \\     
    };
    \foreach[count = \i] \v in {\(\Int{1, 1}_2\), \(\Int{1, 2}\), \(\Int{2, 2}\)}{
      \node[left=2.5ex] at (p-\i-1) {\scriptsize\v};
    }
    \foreach[count = \i] \v in {\(\Int{2, 2}\), \(\Int{1, 2}\)}{
      \node[above=1.5ex] at (p-1-\i) {\scriptsize\v};
    } 
	}
  \bigoplus
  \mattikz{
    \matheader{
      E \\
    };
		\foreach[count = \i] \v in {\(\Int{1, 1}_1\)}{
      \node[left=2.5ex] at (p-\i-1) {\scriptsize\v};
    }
    \foreach[count = \i] \v in {\(\Int{1, 1}_1\)}{
      \node[above=1.5ex] at (p-1-\i) {\scriptsize\v};
    } 
  }
  \bigoplus
  \mattikz{
    \rowempty
    \foreach[count = \i] \v in {\(\Int{1, 1}_2\)}{
      \node[above=1.5ex] at (p-1-\i) {\scriptsize\v};
    } 
  }		  
\end{equation}
which we have expressed as a direct sum of block matrices.

Here, we can extract two indecomposable representations of
$\CL_2(f)$. The identity submatrix $E$ in $(\Int{1,1}_1, \Int{1,1}_1)$ is
\begin{equation*}
	Ef\itoi{1,1,1,1} = \left[
		\begin{array}{cccc}
			1 f\itoi{1,1,1,1} & 0                           & \ldots & 0 \\
			0                           & 1f \itoi{1,1,1,1} & \ldots & 0 \\
			\vdots                      & \vdots                      & \ddots & \vdots \\
			0                           & 0                           & \ldots & 1 f\itoi{1,1,1,1}
		\end{array}
  \right]
\end{equation*}
where
\begin{equation} \label{eq:f1111}
	f_{\Int{1,1}}^{\Int{1,1}} =
	\begin{tikzcd}
		K \rar & 0 \\
		K \rar \uar{\Id_K} & 0 \uar
	\end{tikzcd}
\end{equation}
as in the proof of Lemma~\ref{lemma:homdim}. 

Via the isomorphism functor $F$ in Theorem~\ref{isothm}, the arrow $f\itoi{1,1,1,1}$
in Eq.~\eqref{eq:f1111} can be regarded as the corresponding
representation $F^{-1}(f_{\Int{1,1}}^{\Int{1,1}})$. This is indecomposable.
Thus, $Ef_{\Int{1,1}}^{\Int{1,1}}$ corresponds
to a direct sum of $m$ copies of the representation in
Eq.~\eqref{eq:f1111}, where $m$ is the size of $E$.

The third term in Eq.~\eqref{CLf1} is an empty matrix with
$0$ rows, and represents the arrow $0:\intv[1,1]^{m_1} \rightarrow
0$ in $\rep{\A_n(f)}$, where $m_1$ is the number of $K$-vector columns
in $\Int{1,1}_2$. By the isomorphism, this corresponds to a
direct sum of $m_1$ copies of the indecomposable representation
\begin{equation*}
	\begin{tikzcd}
		0 \rar & 0 \\
		K \rar \uar & 0 \uar
	\end{tikzcd}.
\end{equation*}

Now, the row $\Int{1,1}_1$ and columns $\Int{1,1}_1, \Int{1,1}_2$ in
the block matrix Eq.~\eqref{CLf1} will not affect nor be affected by
subsequent operations, so we hide them from the block matrix.
The unprocessed block $(\Int{1,2}, \Int{1,2})$ is next transformed to
Smith normal form to get
\begin{equation*}
	\mattikz{
    \matheader{
      \n \& \* \\
      \* \& \* \\ 
      \* \& \n \\
		};
		\foreach[count = \i] \v in {\(\Int{1, 1}\), \(\Int{1, 2}\), \(\Int{2, 2}\)}{
      \node[left=2.5ex] at (p-\i-1) {\scriptsize\v};
		}
		\foreach[count = \i] \v in {\(\Int{2, 2}\), \(\Int{1, 2}\)}{
      \node[above=1.5ex] at (p-1-\i) {\scriptsize\v};
		}
	}
  \cong
	\mattikz{
    \matheader{
      \n \& \* \& \* \\
      \* \& \E \& \0 \\ 
      \* \& \0 \& \0 \\ 
      \* \& \n \& \n \\
		};
		\foreach[count = \i] \v in {\(\Int{1, 1}\), \(\Int{1, 2}_1\), \(\Int{1, 2}_2\), \(\Int{2, 2}\)}{
      \node[left=2.5ex] at (p-\i-1) {\scriptsize\v};
		}
		\foreach[count = \i] \v in {\(\Int{2, 2}\), \(\Int{1, 2}_1\), \(\Int{1, 2}_2\)}{
      \node[above=1.5ex] at (p-1-\i) {\scriptsize\v};
		}
	}.
\end{equation*}
We see that $(\Int{1,1}, \Int{1,2}_1)$ is
erasable using $(\Int{1,2}_1, \Int{1,2}_1)$. The checking via
\textproc{row\_erasable} in Algorithm~\ref{algo:rowerasable} proceeds
as follows. While $u = (\Int{1,2}_1, \Int{2,2})$ is a nonzero row
neighbor of $E$, the computed potential side effect is $v_t' =
(\Int{1,1}, \Int{2,2})$. Since $v_t'$ is strongly zero, addition from
row $\Int{1,2}_1$ will not affect it.

Similarly, $(\Int{1,2}_1, \Int{2,2})$ is erasable. After zeroing out erasable blocks, we get
\begin{equation*}
	\mattikz{
    \matheader{
      \n \& \0 \& \* \\
      \* \& \E \& \0 \\ 
      \* \& \0 \& \0 \\ 
      \* \& \n \& \n \\
		};
		\foreach[count = \i] \v in {\(\Int{1, 1}\), \(\Int{1, 2}_1\), \(\Int{1, 2}_2\), \(\Int{2, 2}\)}{
      \node[left=2.5ex] at (p-\i-1) {\scriptsize\v};
		}
		\foreach[count = \i] \v in {\(\Int{2, 2}\), \(\Int{1, 2}_1\), \(\Int{1, 2}_2\)}{
      \node[above=1.5ex] at (p-1-\i) {\scriptsize\v};
		}
	}
  \text{ and then }
	\mattikz{
    \matheader{
      \n \& \0 \& \* \\
      \0 \& \E \& \0 \\ 
      \* \& \0 \& \0 \\ 
      \* \& \n \& \n \\
		};
		\foreach[count = \i] \v in {\(\Int{1, 1}\), \(\Int{1, 2}_1\), \(\Int{1, 2}_2\), \(\Int{2, 2}\)}{
      \node[left=2.5ex] at (p-\i-1) {\scriptsize\v};
		}
		\foreach[count = \i] \v in {\(\Int{2, 2}\), \(\Int{1, 2}_1\), \(\Int{1, 2}_2\)}{
      \node[above=1.5ex] at (p-1-\i) {\scriptsize\v};
		}
	}.
\end{equation*}
The identity submatrix $E$ in $(\Int{1,2}_1, \Int{1,2}_1)$
corresponds to copies of the indecomposable representation
\begin{equation*}
	\begin{tikzcd}
		K \rar & K \\
		K \rar \uar & K \uar
	\end{tikzcd}
\end{equation*}
as direct summands.

Once again we abbreviate the block matrix:
\begin{equation*}
	\mattikz{
    \matheader{
      \n \& \* \\
      \* \& \0 \\ 
      \* \& \n \\
		};
		\foreach[count = \i] \v in {\(\Int{1, 1}\), \(\Int{1, 2}\), \(\Int{2, 2}\)}{
      \node[left=2.5ex] at (p-\i-1) {\scriptsize\v};
		}
		\foreach[count = \i] \v in {\(\Int{2, 2}\), \(\Int{1, 2}\)}{
      \node[above=1.5ex] at (p-1-\i) {\scriptsize\v};
		}
	}
  \text{ and then }
	\mattikz{
    \matheader{
      \n \& \E \& \0 \\
      \n \& \0 \& \0 \\
      \* \& \0 \& \0 \\ 
      \* \& \n \& \n \\
		};
		\foreach[count = \i] \v in {\(\Int{1, 1}_1\), \(\Int{1, 1}_2\), \(\Int{1, 2}\), \(\Int{2, 2}\)}{
      \node[left=2.5ex] at (p-\i-1) {\scriptsize\v};
		}
		\foreach[count = \i] \v in {\(\Int{2, 2}\), \(\Int{1, 2}_1\), \(\Int{1, 2}_2\)}{
      \node[above=1.5ex] at (p-1-\i) {\scriptsize\v};
		}
	}
\end{equation*}
after transforming the next target $(\Int{1,1}, \Int{1,2})$ to Smith normal form.
The identity submatrix in $(\Int{1,1}_1, \Int{1,2}_1)$, the row
$\Int{1,1}_2$, and the column $\Int{1,2}_2$ correspond to copies of the 
indecomposable representations with dimension vectors $\dimv{10}{11}$, $\dimv{10}{00}$, and $\dimv{00}{11}$
respectively, as direct summands.

What remains is the form
  $
	\mattikz{
    \matheader{
      \* \\ 
      \* \\
		};
		\foreach[count = \i] \v in {\(\Int{1, 2}\), \(\Int{2, 2}\)}{
      \node[left=2.5ex] at (p-\i-1) {\scriptsize\v};
		}
		\foreach[count = \i] \v in {\(\Int{2, 2}\)}{
      \node[above=1.5ex] at (p-1-\i) {\scriptsize\v};
		}
	}
  $, from which we get
\begin{equation*}
	\mattikz{
    \matheader{
      \0 \& \* \\ 
      \E \& \0 \\
      \0 \& \0 \\
		};
		\foreach[count = \i] \v in {\(\Int{1, 2}\), \(\Int{2, 2}_1\), \(\Int{2, 2}_2\)}{
      \node[left=2.5ex] at (p-\i-1) {\scriptsize\v};
		}
		\foreach[count = \i] \v in {\(\Int{2, 2}_1\), \(\Int{2, 2}_2\)}{
      \node[above=1.5ex] at (p-1-\i) {\scriptsize\v};
		}
	}
\end{equation*}
after transforming the next target $(\Int{2,2}, \Int{2,2})$ to normal
form, and zeroing out erasable blocks. The identity submatrix
$(\Int{2,2}_1, \Int{2,2}_1)$ and the row $\Int{2,2}_2$ correspond
respectively to copies of the indecomposable representations with dimension vectors $\dimv{01}{01}$, and $\dimv{01}{00}$. 

Abbreviating again, we are left with
$
	\mattikz{
    \matheader{
      \* \\ 
		};
		\foreach[count = \i] \v in {\(\Int{1, 2}\)}{
      \node[left=2.5ex] at (p-\i-1) {\scriptsize\v};
		}
		\foreach[count = \i] \v in {\(\Int{2, 2}\)}{
      \node[above=1.5ex] at (p-1-\i) {\scriptsize\v};
		}
	}
$.
Transforming the last target $(\Int{1,2}, \Int{2,2})$ to normal form yields
\begin{equation*}
	\mattikz{
    \matheader{
      \E \& \0 \\ 
      \0 \& \0 \\ 
		};
		\node[anchor = south east] (p-0-0) at (p-2-1.west |- p-1-1.north west) {};
		\foreach[count = \i] \v in {\(\Int{1, 2}_1\), \(\Int{1, 2}_2\)}{
      \node[left=2.5ex] at (p-\i-1) {\scriptsize\v};
		}
		\foreach[count = \i] \v in {\(\Int{2, 2}_1\), \(\Int{2, 2}_2\)}{
      \node[above=1.5ex] at (p-1-\i) {\scriptsize\v};
		}
	}.
\end{equation*}
The identity submatrix $(\Int{1,2}_1, \Int{2,2}_1)$, the row
$\Int{1,2}_2$, and the column $\Int{2,2}_2$ correspond respectively to
copies of the indecomposable representations with dimension vectors $\dimv{11}{01}$, $\dimv{11}{00}$, and $\dimv{00}{01}$.

It is clear that we have obtained all possible indecomposable
representations of $\CL_2(f)$. This can be confirmed for example by
checking with the Auslander-Reiten quiver of $\CL_2(f)$.
Given a particular persistence module $M$ on
$\CL_2(f)$, the algorithm gives the multiplicities of
each of these indecomposables in an indecomposable decomposition of $M$.


\subsubsection{Case $\CL_3(fb)$}
\label{subsec:clfb}
The input block matrix is given in Example~\ref{example:clfb_mp}.

While the presence of impermissible operations did not cause any noticeable complications in the case of $\CL_2(f)$, in general this is not so. For better readability, we
indicate only the relevant impermissible operations at each step
below. Below, each numbered step corresponds to one pass of the outer while
loop in Algorithm~\ref{algo:main}.

\begin{enumerate}[leftmargin=*]

\item In this step, $v_*$ is $(\Int{1,1}, \Int{1,1})$.
  \begin{enumerate}
  \item Transform $(\Int{1,1}, \Int{1,1})$ to Smith normal form, giving the block matrix
    \[
      \mattikz{
        \matheader{
          \n \& \n \& \* \& \* \& \n \& \E \& \0 \\
          \n \& \n \& \* \& \* \& \n \& \0 \& \0 \\
          \n \& \* \& \n \& \* \& \* \& \n \& \n \\
          \* \& \* \& \* \& \* \& \n \& \n \& \n \\
          \* \& \n \& \* \& \n \& \n \& \n \& \n \\
          \* \& \* \& \n \& \n \& \n \& \n \& \n \\
          \* \& \n \& \n \& \n \& \n \& \n \& \n \\
        };
        \foreach[count = \i] \v in {
          \(\Int{1, 1}_1\), \(\Int{1, 1}_2\),
          \(\Int{3, 3}\), \(\Int{1, 3}\),
          \(\Int{1, 2}\), \(\Int{2, 3}\), \(\Int{2, 2}\)}
        {
          \node[left=2.5ex] at (p-\i-1) {\scriptsize\v};
        }
        \foreach[count = \i] \v in {
          \(\Int{2, 2}\), \(\Int{2, 3}\),
          \(\Int{1, 2}\), \(\Int{1, 3}\),
          \(\Int{3, 3}\), \(\Int{1, 1}_1\), \(\Int{1, 1}_2\)}
        {
          \node[above=1.5ex] at (p-1-\i) {\scriptsize\v};
        }
        \node[right=6pt] (LR) at (p-1-7.east |- p-7-7.south east) {};
      }.
    \]

  \item Zero out the blocks $(\Int{1,1}_1, \Int{1,2})$ and
    $(\Int{1,1}_1, \Int{1,3})$ by additions from the identity submatrix
    at $(\Int{1,1}_1, \Int{1,1}_1)$:
    \[
      \mattikz{
        \matheader{
          \n \& \n \& 0  \&  0 \& \n \& \E \& \0 \\
          \n \& \n \& \* \& \* \& \n \& \0 \& \0 \\
          \n \& \* \& \n \& \* \& \* \& \n \& \n \\
          \* \& \* \& \* \& \* \& \n \& \n \& \n \\
          \* \& \n \& \* \& \n \& \n \& \n \& \n \\
          \* \& \* \& \n \& \n \& \n \& \n \& \n \\
          \* \& \n \& \n \& \n \& \n \& \n \& \n \\
        };
        \foreach[count = \i] \v in {
          \(\Int{1, 1}_1\), \(\Int{1, 1}_2\),
          \(\Int{3, 3}\), \(\Int{1, 3}\),
          \(\Int{1, 2}\), \(\Int{2, 3}\), \(\Int{2, 2}\)}
        {
          \node[left=2.5ex] at (p-\i-1) {\scriptsize\v};
        }
        \foreach[count = \i] \v in {
          \(\Int{2, 2}\), \(\Int{2, 3}\),
          \(\Int{1, 2}\), \(\Int{1, 3}\),
          \(\Int{3, 3}\), \(\Int{1, 1}_1\), \(\Int{1, 1}_2\)}
        {
          \node[above=1.5ex] at (p-1-\i) {\scriptsize\v};
        }
      }.
    \]
    
  \item The identity submatrix $(\Int{1,1}_1, \Int{1,1}_1)$ and the
    columns in $\Int{1,1}_2$ give copies of
    indecomposable representations isomorphic to
    \[
      \begin{tikzcd}
        K \rar      & 0      & 0 \lar \\
        K \rar \uar & 0 \uar & 0 \lar \uar
      \end{tikzcd}
      \text{ and } 
      \begin{tikzcd}
        0 \rar      & 0      & 0 \lar \\
        K \rar \uar & 0 \uar & 0 \lar \uar
      \end{tikzcd}
    \]
    corresponding to the vertices $\dimv{100}{100}$ and
    $\dimv{000}{100}$ in Figure \ref{fig:arfb}.

  \item We are left with
    \[
      \mattikz{
        \matheader{
          \n \& \n \& \* \& \* \& \n \\
          \n \& \* \& \n \& \* \& \* \\
          \* \& \* \& \* \& \* \& \n \\
          \* \& \n \& \* \& \n \& \n \\
          \* \& \* \& \n \& \n \& \n \\
          \* \& \n \& \n \& \n \& \n \\
        };
        \foreach[count = \i] \v in {
          \(\Int{1, 1}\), \(\Int{3, 3}\),
          \(\Int{1, 3}\), \(\Int{1, 2}\),
          \(\Int{2, 3}\), \(\Int{2, 2}\)}
        {
          \node[left=2.5ex] at (p-\i-1) {\scriptsize\v};
        }
        \foreach[count = \i] \v in {
          \(\Int{2, 2}\), \(\Int{2, 3}\),
          \(\Int{1, 2}\), \(\Int{1, 3}\),
          \(\Int{3, 3}\)}
        {
          \node[above=1.5ex] at (p-1-\i) {\scriptsize\v};
        }
      }.
    \]     
  \end{enumerate}

\item We combine the next two steps. Here, $v_*$ is $(\Int{3,3}, \Int{3,3})$, and then...
\item ... $v_*$ is $(\Int{1,3}, \Int{1,3})$.

  The direct summands with dimension vectors $\dimv{001}{001}$ and
  $\dimv{000}{001}$, and then $\dimv{111}{111}$ can be extracted without any extra complications:
  \[
    \mattikz{
      \matheader{
        \n \& \n \& \* \& \* \\
        \n \& \* \& \n \& \* \\
        \* \& \* \& \* \& \* \\
        \* \& \n \& \* \& \n \\
        \* \& \* \& \n \& \n \\
        \* \& \n \& \n \& \n \\
      };
      \foreach[count = \i] \v in {
        \(\Int{1, 1}\), \(\Int{3, 3}\),
        \(\Int{1, 3}\), \(\Int{1, 2}\),
        \(\Int{2, 3}\), \(\Int{2, 2}\)}
      {
        \node[left=2.5ex] at (p-\i-1) {\scriptsize\v};
      }
      \foreach[count = \i] \v in {
        \(\Int{2, 2}\), \(\Int{2, 3}\), \(\Int{1, 2}\), \(\Int{1, 3}\)}
      {
        \node[above=1.5ex] at (p-1-\i) {\scriptsize\v};
      }
    }
   \text{ and then }
    \mattikz{
      \matheader{
        \n \& \n \& \* \& \* \\
        \n \& \* \& \n \& \* \\
        \* \& \* \& \* \& \0 \\
        \* \& \n \& \* \& \n \\
        \* \& \* \& \n \& \n \\
        \* \& \n \& \n \& \n \\
      };
      \foreach[count = \i] \v in {
        \(\Int{1, 1}\), \(\Int{3, 3}\),
        \(\Int{1, 3}\), \(\Int{1, 2}\),
        \(\Int{2, 3}\), \(\Int{2, 2}\)}
      {
        \node[left=2.5ex] at (p-\i-1) {\scriptsize\v};
      }
      \foreach[count = \i] \v in {
        \(\Int{2, 2}\), \(\Int{2, 3}\), \(\Int{1, 2}\), \(\Int{1, 3}\)}
      {
        \node[above=1.5ex] at (p-1-\i) {\scriptsize\v};
      }
    }.
  \]
\setcounter{enumi}{3} 
\item After transforming $v_* = (\Int{3,3}, \Int{1,3})$ to Smith
  normal form, the block matrix is now
  \[
    \mattikz{
      \matheader{
        \n \& \n \& \* \& \* \& \* \\
        \n \& \* \& \n \& \E \& \0 \\
        \n \& \* \& \n \& \0 \& \0 \\
        \* \& \* \& \* \& \0 \& \0 \\
        \* \& \n \& \* \& \n \& \n \\
        \* \& \* \& \n \& \n \& \n \\
        \* \& \n \& \n \& \n \& \n \\
      };
      \foreach[count = \i] \v in {
        \(\Int{1, 1}\), \(\Int{3, 3}_1\), \(\Int{3, 3}_2\),
        \(\Int{1, 3}\), \(\Int{1, 2}\),
        \(\Int{2, 3}\), \(\Int{2, 2}\)}
      {
        \node[left=2.5ex] at (p-\i-1) {\scriptsize\v};
      }
      \foreach[count = \i] \v in {
        \(\Int{2, 2}\), \(\Int{2, 3}\), \(\Int{1, 2}\), \(\Int{1, 3}_1\), \(\Int{1, 3}_2\)}
      {
        \node[above=1.5ex] at (p-1-\i) {\scriptsize\v};
      }
      \node[right=6pt] (LR) at (p-1-5.east |- p-7-5.south east) {};
      \rowarrow{ (LR |- p-2-7.east) }{ (LR |- p-1-7.east)};
    }
    \text{ and then }
    \mattikz{
      \matheader{
        \n \& \n \& \* \& \* \& \* \\
        \n \& \0 \& \n \& \E \& \0 \\
        \n \& \* \& \n \& \0 \& \0 \\
        \* \& \* \& \* \& \0 \& \0 \\
        \* \& \n \& \* \& \n \& \n \\
        \* \& \* \& \n \& \n \& \n \\
        \* \& \n \& \n \& \n \& \n \\
      };
      \foreach[count = \i] \v in {
        \(\Int{1, 1}\), \(\Int{3, 3}_1\), \(\Int{3, 3}_2\),
        \(\Int{1, 3}\), \(\Int{1, 2}\),
        \(\Int{2, 3}\), \(\Int{2, 2}\)}
      {
        \node[left=2.5ex] at (p-\i-1) {\scriptsize\v};
      }
      \foreach[count = \i] \v in {
        \(\Int{2, 2}\), \(\Int{2, 3}\), \(\Int{1, 2}\), \(\Int{1, 3}_1\), \(\Int{1, 3}_2\)}
      {
        \node[above=1.5ex] at (p-1-\i) {\scriptsize\v};
      }
      \node[right=6pt] (LR) at (p-1-5.east |- p-7-5.south east) {};
      \rowarrow{ (LR |- p-2-5.east) }{ (LR |- p-1-5.east)};
    }.
  \]
  Row operations from $\Int{3, 3}_1$ to $\Int{1, 1}$ are impermissible, and so the only candidate $E$ cannot be used to zero out $(\Int{1,1},\Int{1,3}_1)$. Block $(\Int{1,1},\Int{1,3}_1)$ is therefore not erasable. The block $(\Int{3, 3}_1, \Int{2, 3})$ however, is erasable and is zeroed out. No
  direct summands are identified at this step.

\item Next, $v_*$ is $(\Int{1,1}, \Int{1,3}_2)$ and direct summands
  $\dimv{100}{111}$ and $\dimv{000}{111}$ can be extracted. We are left with the block matrix form:
  \[
    \mattikz{
      \matheader{
        \n \& \n \& \* \& \* \\
        \n \& \0 \& \n \& \E \\
        \n \& \* \& \n \& \0 \\
        \* \& \* \& \* \& \0 \\
        \* \& \n \& \* \& \n \\
        \* \& \* \& \n \& \n \\
        \* \& \n \& \n \& \n \\
      };
      \foreach[count = \i] \v in {
        \(\Int{1, 1}\), \(\Int{3, 3}_1\), \(\Int{3, 3}_2\),
        \(\Int{1, 3}\), \(\Int{1, 2}\),
        \(\Int{2, 3}\), \(\Int{2, 2}\)}
      {
        \node[left=2.5ex] at (p-\i-1) {\scriptsize\v};
      }
      \foreach[count = \i] \v in {
        \(\Int{2, 2}\), \(\Int{2, 3}\), \(\Int{1, 2}\), \(\Int{1, 3}\)}
      {
        \node[above=1.5ex] at (p-1-\i) {\scriptsize\v};
      }
      \node[right=6pt] (LR) at (p-1-4.east |- p-7-4.south east) {};
      \rowarrow{ (LR |- p-2-4.east) }{ (LR |- p-1-4.east)};
    }.
  \] 

\item
  After transforming $v_* = (\Int{1,1}, \Int{1,3})$ to Smith normal
  form, the column neighbor $(\Int{3,3}_1, \Int{1,3})$ may no longer be the identity. After \textproc{row\_fix}
  it is transformed back to an identity. The block matrix is now in
  the following form:
  \[
    \mattikz{
      \matheader{
        \n \& \n \& \* \& \E \& \0 \\
        \n \& \n \& \* \& \0 \& \0 \\
        \n \& \0 \& \n \& \E \& \0 \\
        \n \& \0 \& \n \& \0 \& \E \\
        \n \& \* \& \n \& \0 \& \0 \\
        \* \& \* \& \* \& \0 \& \0 \\
        \* \& \n \& \* \& \n \& \n \\
        \* \& \* \& \n \& \n \& \n \\
        \* \& \n \& \n \& \n \& \n \\
      };
      \foreach[count = \i] \v in {
        \(\Int{1, 1}_1\), \(\Int{1, 1}_2\),
        \(\Int{3, 3}_1\), \(\Int{3, 3}_2\), \(\Int{3, 3}_3\),
        \(\Int{1, 3}\), \(\Int{1, 2}\),
        \(\Int{2, 3}\), \(\Int{2, 2}\)}
      {
        \node[left=2.5ex] at (p-\i-1) {\scriptsize\v};
      }
      \foreach[count = \i] \v in {
        \(\Int{2, 2}\), \(\Int{2, 3}\),
        \(\Int{1, 2}\), \(\Int{1, 3}_1\), \(\Int{1, 3}_2\)}{
        \node[above=1.5ex] at (p-1-\i) {\scriptsize\v};
      }
      \node[right=6pt] (LR) at (p-1-5.east |- p-9-5.south east) {};
      \rowarrow{ (LR |- p-3-5.east) }{ (LR |- p-1-5.east)};
    }
  \]
  after repartitioning to isolate identity submatrices into their own blocks.

  After zeroing out $(\Int{1,1}_1, \Int{1,2})$, direct summands with
  dimension vectors $\dimv{101}{111}$ and $\dimv{001}{111}$ can be
  extracted, as follows. We get the decomposition
  
  \[
    \mattikz{
      \matheader{
        \n \& \n \& \0 \& \E \& \0 \\
        \n \& \n \& \* \& \0 \& \0 \\
        \n \& \0 \& \n \& \E \& \0 \\
        \n \& \0 \& \n \& \0 \& \E \\
        \n \& \* \& \n \& \0 \& \0 \\
        \* \& \* \& \* \& \0 \& \0 \\
        \* \& \n \& \* \& \n \& \n \\
        \* \& \* \& \n \& \n \& \n \\
        \* \& \n \& \n \& \n \& \n \\
      };
      \foreach[count = \i] \v in {
        \(\Int{1, 1}_1\), \(\Int{1, 1}_2\),
        \(\Int{3, 3}_1\), \(\Int{3, 3}_2\), \(\Int{3, 3}_3\),
        \(\Int{1, 3}\), \(\Int{1, 2}\),
        \(\Int{2, 3}\), \(\Int{2, 2}\)}
      {
        \node[left=2.5ex] at (p-\i-1) {\scriptsize\v};
      }
      \foreach[count = \i] \v in {
        \(\Int{2, 2}\), \(\Int{2, 3}\),
        \(\Int{1, 2}\), \(\Int{1, 3}_1\), \(\Int{1, 3}_2\)}{
        \node[above=1.5ex] at (p-1-\i) {\scriptsize\v};
      }
    }
    \cong    
    \mattikz{
      \matheader{
        \n \& \n \& \* \\
        \n \& \* \& \n \\
        \* \& \* \& \* \\
        \* \& \n \& \* \\
        \* \& \* \& \n \\
        \* \& \n \& \n \\
      };
      \foreach[count = \i] \v in {
        \(\Int{1, 1}_2\), \(\Int{3, 3}_3\),
        \(\Int{1, 3}\), \(\Int{1, 2}\), \(\Int{2, 3}\), \(\Int{2, 2}\)}{
        \node[left=2.5ex] at (p-\i-1) {\scriptsize\v};
      }
      \foreach[count = \i] \v in {
        \(\Int{2, 2}\), \(\Int{2, 3}\), \(\Int{1, 2}\)}
      {
        \node[above=1.5ex] at (p-1-\i) {\scriptsize\v};
      }
    }
    \bigoplus
    \mattikz{
      \matheader{
        E \\
        E \\
      };
      \foreach[count = \i] \v in {\(\Int{1, 1}_1\), \(\Int{3,3}_1\)}{
        \node[left=2.5ex] at (p-\i-1) {\scriptsize\v};
      }
      \foreach[count = \i] \v in {\(\Int{1, 3}_1\)}{
        \node[above=1.5ex] at (p-1-\i) {\scriptsize\v};
      } 
    }
    \bigoplus
    \mattikz{
      \matheader{
        E \\
      };
      \foreach[count = \i] \v in {\(\Int{3, 3}_2\)}{
        \node[left=2.5ex] at (p-\i-1) {\scriptsize\v};
      }
      \foreach[count = \i] \v in {\(\Int{1, 3}_2\)}{
        \node[above=1.5ex] at (p-1-\i) {\scriptsize\v};
      } 
    },
  \]
  where the second term corresponds to $\dimv{101}{111}$, and the
  third term to $\dimv{001}{111}$. The first term is sent to the next step.

\item We combine the next two steps. First, $v_*$ is $(\Int{1,2},
  \Int{1,2})$ and $\dimv{110}{110}$ is extracted.
\item Subsequently, $v_*$ is $(\Int{1,3}, \Int{1,2})$. These two steps yield the block matrix forms
  \[
    \mattikz{
      \matheader{
        \n \& \n \& \* \\
        \n \& \* \& \n \\
        \* \& \* \& \* \\
        \* \& \n \& \0 \\
        \* \& \* \& \n \\
        \* \& \n \& \n \\
      };
      \foreach[count = \i] \v in {\(\Int{1, 1}\), \(\Int{3, 3}\), \(\Int{1, 3}\), \(\Int{1, 2}\), \(\Int{2, 3}\), \(\Int{2, 2}\)}{
        \node[left=2.5ex] at (p-\i-1) {\scriptsize\v};
      }
      \foreach[count = \i] \v in {\(\Int{2, 2}\), \(\Int{2, 3}\), \(\Int{1, 2}\)}{
        \node[above=1.5ex] at (p-1-\i) {\scriptsize\v};
      }
    }
    \text{ and then }
    \mattikz{
      \matheader{
        \n \& \n \& \0 \& \* \\
        \n \& \* \& \n \& \n \\
        \0 \& \* \& \E \& \0 \\
        \* \& \* \& \0 \& \0 \\
        \* \& \n \& \0 \& \0 \\
        \* \& \* \& \n \& \n \\
        \* \& \n \& \n \& \n \\
      };
      \foreach[count = \i] \v in {
        \(\Int{1, 1}\), \(\Int{3, 3}\),
        \(\Int{1, 3}_1\), \(\Int{1, 3}_2\), \(\Int{1, 2}\), \(\Int{2,
          3}\), \(\Int{2, 2}\)}{
        \node[left=2.5ex] at (p-\i-1) {\scriptsize\v};
      }
      \foreach[count = \i] \v in {
        \(\Int{2, 2}\), \(\Int{2, 3}\), \(\Int{1, 2}_1\), \(\Int{1, 2}_2\)}
      {
        \node[above=1.5ex] at (p-1-\i) {\scriptsize\v};
      }
      \node[right=6pt] (LR) at (p-1-4.east |- p-7-4.south east) {};
      \colarrow{ (p-7-3.south |- LR) }{ (p-7-2.south |- LR)};
    }.
  \]
  Note that column operations from $\Int{1,2}$ to $\Int{2,3}$ are
  impermissible.

\item Combining steps again for brevity, $v_*$ is $(\Int{1,1}, \Int{1,2}_2)$ and then... 
\item ... $v_*$ is $(\Int{2,3}, \Int{2,3})$.

  First $\dimv{100}{110}$, $\dimv{100}{000}$, and $\dimv{000}{110}$ are
  extracted, and then $\dimv{011}{011}$ is extracted. The matrix form
  becomes
  \[
    \mattikz{
      \matheader{
        \n \& \* \& \n \\
        \0 \& \* \& \E \\
        \* \& \* \& \0 \\
        \* \& \n \& \0 \\
        \* \& \* \& \n \\
        \* \& \n \& \n \\
      };
      \foreach[count = \i] \v in {\(\Int{3, 3}\), \(\Int{1, 3}_1\), \(\Int{1, 3}_2\), \(\Int{1, 2}\), \(\Int{2, 3}\), \(\Int{2, 2}\)}{
        \node[left=2.5ex] at (p-\i-1) {\scriptsize\v};
      }
      \foreach[count = \i] \v in {\(\Int{2, 2}\), \(\Int{2, 3}\), \(\Int{1, 2}\)}{
        \node[above=1.5ex] at (p-1-\i) {\scriptsize\v};
      }
      \node[right=6pt] (LR) at (p-1-3.east |- p-6-3.south east) {};
      \colarrow{ (p-6-3.south |- LR) }{ (p-6-2.south |- LR)};
    }
  \text{ and then }
    \mattikz{
      \matheader{
        \n \& \* \& \n \\
        \0 \& \* \& \E \\
        \* \& \* \& \0 \\
        \* \& \n \& \0 \\
        \* \& \0 \& \n \\
        \* \& \n \& \n \\
      };
      \foreach[count = \i] \v in {\(\Int{3, 3}\), \(\Int{1, 3}_1\), \(\Int{1, 3}_2\), \(\Int{1, 2}\), \(\Int{2, 3}\), \(\Int{2, 2}\)}{
        \node[left=2.5ex] at (p-\i-1) {\scriptsize\v};
      }
      \foreach[count = \i] \v in {\(\Int{2, 2}\), \(\Int{2, 3}\), \(\Int{1, 2}\)}{
        \node[above=1.5ex] at (p-1-\i) {\scriptsize\v};
      }
      \node[right=6pt] (LR) at (p-1-3.east |- p-6-3.south east) {};
      \colarrow{ (p-6-3.south |- LR) }{ (p-6-2.south |- LR)};
    }.
  \]

\item Here, $v_*$ is $(\Int{1,3}_2, \Int{2,3})$. After transforming
  $v_*$ to Smith normal form, the matrix is of the form
  \[
    \mattikz{
      \matheader{
        \n \& \* \& \* \& \n \\
        \0 \& \* \& \* \& \E \\
        \* \& \E \& \0 \& \0 \\
        \* \& \0 \& \0 \& \0 \\
        \* \& \n \& \n \& \0 \\
        \* \& \0 \& \0 \& \n \\
        \* \& \n \& \n \& \n \\
      };
      \foreach[count = \i] \v in {\(\Int{3, 3}\), \(\Int{1, 3}_1\), \(\Int{1, 3}_2\), \(\Int{1, 3}_3\), \(\Int{1, 2}\), \(\Int{2, 3}\), \(\Int{2, 2}\)}{
        \node[left=2.5ex] at (p-\i-1) {\scriptsize\v};
      }
      \foreach[count = \i] \v in {\(\Int{2, 2}\), \(\Int{2, 3}_1\), \(\Int{2, 3}_2\), \(\Int{1, 2}\)}{
        \node[above=1.5ex] at (p-1-\i) {\scriptsize\v};
      }
      \node[right=6pt] (LR) at (p-1-4.east |- p-7-4.south east) {};
      \colarrow{ (p-7-4.south |- LR) }{ (p-7-3.south |- LR)};
      \colarrow{ (p-7-4.south |- LR) }{ (p-7-2.south |- LR)};
    }.
  \]
  First of all, note that $(\Int{3,3}, \Int{2,3}_1)$ is erasable, without any
  concerns of side effects.

  Moreover, $v_t = (\Int{1,3}_1, \Int{2,3}_1)$ is also erasable. Zeroing it out
  by additions from the identity $E$ at
  $(\Int{1,3}_2, \Int{2,3}_1)$ may cause the $0$ block at
  $v_t' = (\Int{1,3}_1, \Int{2,2})$ to become nonzero, but $v_t'$ can be
  zeroed out again by additions from $(\Int{1,3}_1, \Int{1,2})$. In
  other words, the side effect $v_t'$ is erasable, and thus $v_t$ is, too. This illustrates
  the idea behind Algorithm \ref{algo:rowerasable} and the recursive
  checking of erasability. Similarly, $(\Int{1,3}_2, \Int{2,2})$ is also erasable.

  We are thus able to extract $\dimv{111}{011}$, leaving the form:  
  \[
    \mattikz{
      \matheader{
        \n \& \* \& \n \\
        \0 \& \* \& \E \\
        \* \& \0 \& \0 \\
        \* \& \n \& \0 \\
        \* \& \0 \& \n \\
        \* \& \n \& \n \\
      };
      \foreach[count = \i] \v in {\(\Int{3, 3}\), \(\Int{1, 3}_1\),
        \(\Int{1, 3}_2\), \(\Int{1, 2}\), \(\Int{2, 3}\), \(\Int{2,
          2}\)}{
        \node[left=2.5ex] at (p-\i-1) {\scriptsize\v};
      }
      \foreach[count = \i] \v in {\(\Int{2, 2}\), \(\Int{2, 3}\), \(\Int{1, 2}\)}{
        \node[above=1.5ex] at (p-1-\i) {\scriptsize\v};
      }
      \node[right=6pt] (LR) at (p-1-3.east |- p-6-3.south east) {};
      \colarrow{ (p-6-3.south |- LR) }{ (p-6-2.south |- LR)};
    }.
  \]

\item 
  The procedures from this step on are similar to the ones we have
  done, and direct summands with dimension vectors
  $\dimv{111}{110}$,
  $\dimv{111}{121}$,
  $\dimv{001}{011}$,
  $\dimv{001}{000}$,
  $\dimv{000}{011}$,
  $\dimv{010}{010}$,
  $\dimv{010}{000}$,
  $\dimv{011}{000}$,
  $\dimv{110}{010}$,
  $\dimv{111}{010}$,
  $\dimv{111}{000}$,
  $\dimv{000}{010}$,
  $\dimv{011}{010}$,
  $\dimv{121}{010}$, and
  $\dimv{110}{000}$ will be
  extracted.

  The dimension vector $\dimv{111}{121}$ comes from the direct summand
  $
    \mattikz{
      \matheader{
        E \& E\\ 
      };
      \foreach[count = \i] \v in {\(\Int{1, 3}\)}{
        \node[left=2.5ex] at (p-\i-1) {\scriptsize\v};
      }
      \foreach[count = \i] \v in {\(\Int{2, 3}\), \(\Int{1, 2}\)}{
        \node[above=1.5ex] at (p-1-\i) {\scriptsize\v};
      }
      \node[right=6pt] (LR) at (p-1-2.east |- p-1-2.south east) {};
      \draw[<->, red, postaction={decorate}] (p-1-2.south |- LR)  to[out=210,in=330] (p-1-1.south |- LR);
      
    }.
  $
  This is $m$ copies of the arrow
  $
  \mat{
    \begin{array}{cc}
      f\itoi{2,3,1,3} & f\itoi{1,2,1,3}
    \end{array}
  }
  :
  \intv[2,3] \oplus \intv[1,2] \rightarrow \intv[1,3]
  $,
  where $m$ is the size of the $E$.
  The dimension vector $\dimv{121}{010}$ similarly comes from a summand with two identity blocks that cannot zero out each other. Via the isomorphism $F$ in Theorem~\ref{isothm}, these summands corresponds to the representations given in Eq.~\eqref{eq:dim2_indecs}.
\end{enumerate}


\subsubsection{Case $\CL_4(fff)$}
\label{subsec:clfff}

We have not yet seen an identity block declared erasable in Algorithm~\ref{alg:erasable}. For
$\CL_4(fff)$, this occurs while working with unprocessed blocks in the
column $\Int{3, 3}$. Below, we quickly go through
the procedures leading up to this occurrence.


The input block matrix is of the form
\[
	\mattikz{
		\matheader{
      \n \& \n \& \n \& \n \& \n \& \* \& \n \& \* \& \* \& \* \\
      \n \& \n \& \n \& \* \& \* \& \* \& \* \& \* \& \* \& \n \\
      \n \& \* \& \* \& \* \& \* \& \* \& \n \& \* \& \n \& \n \\
      \n \& \n \& \n \& \* \& \* \& \n \& \* \& \n \& \n \& \n \\
      \* \& \* \& \n \& \* \& \n \& \* \& \n \& \n \& \n \& \n \\
      \n \& \* \& \* \& \* \& \* \& \n \& \n \& \n \& \n \& \n \\
      \* \& \* \& \n \& \* \& \n \& \n \& \n \& \n \& \n \& \n \\
      \n \& \* \& \* \& \n \& \n \& \n \& \n \& \n \& \n \& \n \\
      \* \& \* \& \n \& \n \& \n \& \n \& \n \& \n \& \n \& \n \\
      \* \& \n \& \n \& \n \& \n \& \n \& \n \& \n \& \n \& \n \\
		};
		\node[anchor = south east] (p-0-0) at (p-2-1.west |- p-1-1.north west) {};
		\foreach[count = \i] \v in {\(\Int{1, 1}\), \(\Int{1, 2}\),
      \(\Int{1, 3}\), \(\Int{2, 2}\), \(\Int{1, 4}\), \(\Int{2, 3}\),
      \(\Int{2, 4}\), \(\Int{3, 3}\), \(\Int{3, 4}\), \(\Int{4, 4}\)}{
      \node[left=2.5ex] at (p-\i-1) {\scriptsize\v}; }
		\foreach[count = \i] \v in {\(\Int{4, 4}\), \(\Int{3, 4}\),
      \(\Int{3, 3}\), \(\Int{2, 4}\), \(\Int{2, 3}\), \(\Int{1, 4}\),
      \(\Int{2, 2}\), \(\Int{1, 3}\), \(\Int{1, 2}\), \(\Int{1, 1}\)}{
      \node[above=1.5ex] at (p-1-\i) {\scriptsize\v}; }
	}
\]
in general. We have chosen not to display the impermissible
operations here. For the steps similar to ones already done in
previous cases, we only provide the resulting block matrix form after
sequences of operations. Each numbered item below expresses the result
after a sequence of steps involving $v_*$ taken from a particular
column.

\begin{enumerate}[leftmargin=*]
  
\item By procedures on column $\Int{1, 1}$, direct summands with dimension vectors 
  $\dimv{1000}{1000}$ and
  $\dimv{0000}{1000}$
  can be extracted, leaving us with the block matrix form: 
  \[
    \mattikz{
      \matheader{
        \n \& \n \& \n \& \n \& \n \& \* \& \n \& \* \& \* \\
        \n \& \n \& \n \& \* \& \* \& \* \& \* \& \* \& \* \\
        \n \& \* \& \* \& \* \& \* \& \* \& \n \& \* \& \n \\
        \n \& \n \& \n \& \* \& \* \& \n \& \* \& \n \& \n \\
        \* \& \* \& \n \& \* \& \n \& \* \& \n \& \n \& \n \\
        \n \& \* \& \* \& \* \& \* \& \n \& \n \& \n \& \n \\
        \* \& \* \& \n \& \* \& \n \& \n \& \n \& \n \& \n \\
        \n \& \* \& \* \& \n \& \n \& \n \& \n \& \n \& \n \\
        \* \& \* \& \n \& \n \& \n \& \n \& \n \& \n \& \n \\
        \* \& \n \& \n \& \n \& \n \& \n \& \n \& \n \& \n \\
      };
      \node[anchor = south east] (p-0-0) at (p-2-1.west |- p-1-1.north west) {};
      \foreach[count = \i] \v in {\(\Int{1, 1}\), \(\Int{1, 2}\),
        \(\Int{1, 3}\), \(\Int{2, 2}\), \(\Int{1, 4}\), \(\Int{2, 3}\),
        \(\Int{2, 4}\), \(\Int{3, 3}\), \(\Int{3, 4}\), \(\Int{4, 4}\)}{
        \node[left=2.5ex] at (p-\i-1) {\scriptsize\v}; }
      \foreach[count = \i] \v in {\(\Int{4, 4}\), \(\Int{3, 4}\),
        \(\Int{3, 3}\), \(\Int{2, 4}\), \(\Int{2, 3}\), \(\Int{1, 4}\),
        \(\Int{2, 2}\), \(\Int{1, 3}\), \(\Int{1, 2}\)}{
        \node[above=1.5ex] at (p-1-\i) {\scriptsize\v}; }
    }.
  \]

\item Next, procedures on column $\Int{1, 2}$ yield   
  $\dimv{1100}{1100}$,
  $\dimv{1000}{1100}$, and
  $\dimv{0000}{1100}$, 
  giving us the form:
  \[
    \mattikz{
      \matheader{
        \n \& \n \& \n \& \n \& \n \& \* \& \n \& \* \\
        \n \& \n \& \n \& \* \& \* \& \* \& \* \& \* \\
        \n \& \* \& \* \& \* \& \* \& \* \& \n \& \* \\
        \n \& \n \& \n \& \* \& \* \& \n \& \* \& \n \\
        \* \& \* \& \n \& \* \& \n \& \* \& \n \& \n \\
        \n \& \* \& \* \& \* \& \* \& \n \& \n \& \n \\
        \* \& \* \& \n \& \* \& \n \& \n \& \n \& \n \\
        \n \& \* \& \* \& \n \& \n \& \n \& \n \& \n \\
        \* \& \* \& \n \& \n \& \n \& \n \& \n \& \n \\
        \* \& \n \& \n \& \n \& \n \& \n \& \n \& \n \\
      };
      \node[anchor = south east] (p-0-0) at (p-2-1.west |- p-1-1.north west) {};
      \foreach[count = \i] \v in {\(\Int{1, 1}\), \(\Int{1, 2}\),
        \(\Int{1, 3}\), \(\Int{2, 2}\), \(\Int{1, 4}\), \(\Int{2, 3}\),
        \(\Int{2, 4}\), \(\Int{3, 3}\), \(\Int{3, 4}\), \(\Int{4, 4}\)}{
        \node[left=2.5ex] at (p-\i-1) {\scriptsize\v};
      }
      \foreach[count = \i] \v in {\(\Int{4, 4}\), \(\Int{3, 4}\),
        \(\Int{3, 3}\), \(\Int{2, 4}\), \(\Int{2, 3}\), \(\Int{1, 4}\),
        \(\Int{2, 2}\), \(\Int{1, 3}\)}{
        \node[above=1.5ex] at (p-1-\i) {\scriptsize\v};
      }
    }.
  \]

\item Procedures on column $\Int{1, 3}$ extract
  $\dimv{1110}{1110}$,
  $\dimv{1000}{1110}$, and
  $\dimv{0000}{1110}$. The block
  matrix is now  
  \[
    \mattikz{
      \matheader{
        \n \& \n \& \n \& \n \& \n \& \* \& \n \& \0 \\
        \n \& \n \& \n \& \0 \& \0 \& \0 \& \* \& \E \\
        \n \& \n \& \n \& \* \& \* \& \* \& \* \& \0 \\
        \n \& \* \& \* \& \* \& \* \& \* \& \n \& \0 \\
        \n \& \n \& \n \& \* \& \* \& \n \& \* \& \n \\
        \* \& \* \& \n \& \* \& \n \& \* \& \n \& \n \\
        \n \& \* \& \* \& \* \& \* \& \n \& \n \& \n \\
        \* \& \* \& \n \& \* \& \n \& \n \& \n \& \n \\
        \n \& \* \& \* \& \n \& \n \& \n \& \n \& \n \\
        \* \& \* \& \n \& \n \& \n \& \n \& \n \& \n \\
        \* \& \n \& \n \& \n \& \n \& \n \& \n \& \n \\
      };
      \node[anchor = south east] (p-0-0) at (p-2-1.west |- p-1-1.north west) {};
      \foreach[count = \i] \v in {\(\Int{1, 1}\), \(\Int{1, 2}_1\),
        \(\Int{1, 2}_2\), \(\Int{1, 3}\), \(\Int{2, 2}\), \(\Int{1, 4}\),
        \(\Int{2, 3}\), \(\Int{2, 4}\), \(\Int{3, 3}\), \(\Int{3, 4}\), \(\Int{4, 4}\)}{
        \node[left=2.5ex] at (p-\i-1) {\scriptsize\v};
      }
      \foreach[count = \i] \v in {\(\Int{4, 4}\), \(\Int{3, 4}\),
        \(\Int{3, 3}\), \(\Int{2, 4}\), \(\Int{2, 3}\), \(\Int{1, 4}\),
        \(\Int{2, 2}\), \(\Int{1, 3}\)}{
        \node[above=1.5ex] at (p-1-\i) {\scriptsize\v};
      }
    }
  \]
  where we note that column operations from $\Int{1,3}$ to $\Int{2,2}$
  are impermissible.

\item Procedures on column $\Int{2, 2}$ yield $\dimv{0100}{0100}$,
  $\dimv{1100}{1110}$, $\dimv{1100}{1210}$, and $\dimv{0000}{0100}$.
  Similar to what we have seen before, $\dimv{1100}{1210}$ arises from
  the direct summand
  $
    \mattikz{
      \matheader{
        E \& E\\ 
      };
      \foreach[count = \i] \v in {\(\Int{1, 2}_1\)}{
        \node[left=2.5ex] at (p-\i-1) {\scriptsize\v};
      }
      \foreach[count = \i] \v in {\(\Int{2, 2}\), \(\Int{1, 3}\)}{
        \node[above=1.5ex] at (p-1-\i) {\scriptsize\v};
      }
      \node[right=6pt] (LR) at (p-1-2.east |- p-1-2.south east) {};
      \colarrow{ (p-1-2.south |- LR) }{ (p-1-1.south |- LR)};
    }
  $
  that can be obtained after the prescribed operations.
  We then obtain the matrix form:
  \begin{equation*}
    \mattikz{
      \matheader{
        \n \& \n \& \n \& \n \& \n \& \* \& \n \\
        \n \& \n \& \n \& \0 \& \0 \& \* \& \E \\
        \n \& \n \& \n \& \* \& \* \& \* \& \0 \\
        \n \& \* \& \* \& \* \& \* \& \* \& \n \\
        \n \& \n \& \n \& \* \& \* \& \n \& \0 \\
        \* \& \* \& \n \& \* \& \n \& \* \& \n \\
        \n \& \* \& \* \& \* \& \* \& \n \& \n \\
        \* \& \* \& \n \& \* \& \n \& \n \& \n \\
        \n \& \* \& \* \& \n \& \n \& \n \& \n \\
        \* \& \* \& \n \& \n \& \n \& \n \& \n \\
        \* \& \n \& \n \& \n \& \n \& \n \& \n \\
      };
      \node[anchor = south east] (p-0-0) at (p-2-1.west |- p-1-1.north west) {};
      \foreach[count = \i] \v in {\(\Int{1, 1}\), \(\Int{1, 2}_1\), \(\Int{1, 2}_2\), \(\Int{1, 3}\), \(\Int{2, 2}\), \(\Int{1, 4}\), \(\Int{2, 3}\), \(\Int{2, 4}\), \(\Int{3, 3}\), \(\Int{3, 4}\), \(\Int{4, 4}\)}{
        \node[left=2.5ex] at (p-\i-1) {\scriptsize\v};
      }
      \foreach[count = \i] \v in {\(\Int{4, 4}\), \(\Int{3, 4}\), \(\Int{3, 3}\), \(\Int{2, 4}\), \(\Int{2, 3}\), \(\Int{1, 4}\), \(\Int{2, 2}\)}{
        \node[above=1.5ex] at (p-1-\i) {\scriptsize\v};
      }
    }.
  \end{equation*}

\item Working on column $\Int{1, 4}$ next, we extract
  $\dimv{1111}{1111}$, 
  $\dimv{1100}{1211}$, 
  $\dimv{1100}{0100}$, 
  $\dimv{1000}{1111}$, 
  $\dimv{0000}{1111}$, and 
  $\dimv{1000}{0000}$.
  The matrix is now of the form
  \begin{equation*}
    \mattikz{
      \matheader{
        \n \& \n \& \n \& \0 \& \* \& \0 \& \E \\
        \n \& \n \& \n \& \* \& \* \& \0 \& \0 \\
        \n \& \0 \& \* \& \0 \& \* \& \E \& \0 \\
        \n \& \* \& \* \& \* \& \* \& \0 \& \0 \\
        \n \& \n \& \n \& \* \& \* \& \n \& \n \\
        \* \& \* \& \n \& \* \& \n \& \0 \& \0 \\
        \n \& \* \& \* \& \* \& \* \& \n \& \n \\
        \* \& \* \& \n \& \* \& \n \& \n \& \n \\
        \n \& \* \& \* \& \n \& \n \& \n \& \n \\
        \* \& \* \& \n \& \n \& \n \& \n \& \n \\
        \* \& \n \& \n \& \n \& \n \& \n \& \n \\
      };
      \node[anchor = south east] (p-0-0) at (p-2-1.west |- p-1-1.north west) {};
      \foreach[count = \i] \v in {\(\Int{1, 2}_1\), \(\Int{1, 2}_2\), \(\Int{1, 3}_1\), \(\Int{1, 3}_2\), \(\Int{2, 2}\), \(\Int{1, 4}\), \(\Int{2, 3}\), \(\Int{2, 4}\), \(\Int{3, 3}\), \(\Int{3, 4}\), \(\Int{4, 4}\)}{
        \node[left=2.5ex] at (p-\i-1) {\scriptsize\v};
      }
      \foreach[count = \i] \v in {\(\Int{4, 4}\), \(\Int{3, 4}\), \(\Int{3, 3}\), \(\Int{2, 4}\), \(\Int{2, 3}\), \(\Int{1, 4}_1\), \(\Int{1, 4}_2\)}{
        \node[above=1.5ex] at (p-1-\i) {\scriptsize\v};
      }
    }.
  \end{equation*}

\item From column $\Int{2, 3}$, 
  $\dimv{1110}{0110}$, 
  $\dimv{0110}{0110}$,
  $\dimv{1100}{0110}$, 
  $\dimv{1100}{1221}$, 
  $\dimv{1100}{1111}$, 
  $\dimv{0000}{0110}$, 
  $\dimv{1110}{1221}$, 
  $\dimv{2210}{1221}$, 
  $\dimv{1210}{1221}$, and
  $\dimv{0100}{0110}$ are extracted.
  The summands with dimension vectors $\dimv{2210}{1221}$ and $\dimv{1210}{1221}$ involve three identity blocks. In particular, $\dimv{2210}{1221}$ corresponds to a direct summand
  $
    \mattikz{
      \matheader{
        \E \& \0 \\
        \E \& \E \\
      };
      \foreach[count = \i] \v in { \(\Int{1, 2}\), \(\Int{1, 3}\)} {
        \node[left=2.5ex] at (p-\i-1) {\scriptsize\v};
      }
      \foreach[count = \i] \v in {\(\Int{2, 3}\), \(\Int{1, 4}\)}{
        \node[above=1.5ex] at (p-1-\i) {\scriptsize\v};
      }
    }
  $.
  This is isomorphic to copies of the persistence module
  \[
    \begin{tikzcd}[ampersand replacement=\&]
      K^2 \rar{\Id} \&
      K^2 \rar{\mat{\begin{smallmatrix}0&1\end{smallmatrix}}} \&
      K \rar \&
      0
      \\
      K \rar{\mat{\begin{smallmatrix}0\\1\end{smallmatrix}}}
      \uar{\mat{\begin{smallmatrix}0\\1\end{smallmatrix}}} \&
      K^2 \rar{\Id}
      \uar[swap]{\mat{\begin{smallmatrix}1&0\\1&1\end{smallmatrix}}}
      \&
      K^2 \rar{\mat{\begin{smallmatrix}0&1\end{smallmatrix}}}
      \uar[swap]{\mat{\begin{smallmatrix}1&1\end{smallmatrix}}} \&
      K \uar
	\end{tikzcd}
  \]
  by a choice of basis. Similarly, $\dimv{1210}{1221}$ comes from a direct summand
  $
    \mattikz{
      \matheader{
        \E \& \E \\
        \E \& \n \\
      };
      \foreach[count = \i] \v in { \(\Int{1, 3}\), \(\Int{2, 2}\)} {
        \node[left=2.5ex] at (p-\i-1) {\scriptsize\v};
      }
      \foreach[count = \i] \v in {\(\Int{2, 3}\), \(\Int{1, 4}\)}{
        \node[above=1.5ex] at (p-1-\i) {\scriptsize\v};
      }
    }
  $
  which is copies of the persistence module
  \[
    \begin{tikzcd}[ampersand replacement=\&]
      K \rar{\mat{\begin{smallmatrix}1\\0\end{smallmatrix}}} \&
      K^2 \rar{\mat{\begin{smallmatrix}1&0\end{smallmatrix}}} \&
      K \rar \&
      0
      \\
      K \rar{\mat{\begin{smallmatrix}0\\1\end{smallmatrix}}}
      \uar{\Id} \&
      K^2 \rar{\Id}
      \uar[swap]{\mat{\begin{smallmatrix}1&1\\1&0\end{smallmatrix}}} \&
      K^2 \rar{\mat{\begin{smallmatrix}0&1\end{smallmatrix}}}
      \uar[swap]{\mat{\begin{smallmatrix}1&1\end{smallmatrix}}} \&
      K \uar
    \end{tikzcd}.
  \]

  We caution the reader that several identity blocks will appear and stay for some iterations before being extracted. The matrix form below is the result of all the operations taken in this step.

  \begin{equation*}
    \mattikz{
      \matheader{
        \n \& \n \& \n \& \* \& \0 \& \0 \\
        \n \& \0 \& \* \& \0 \& \0 \& \E \\
        \n \& \0 \& \0 \& \* \& \E \& \0 \\
        \n \& \* \& \* \& \* \& \0 \& \0 \\
        \n \& \n \& \n \& \0 \& \E \& \n \\
        \n \& \n \& \n \& \* \& \0 \& \n \\
        \* \& \* \& \n \& \* \& \n \& \0 \\
        \n \& \* \& \* \& \* \& \0 \& \n \\
        \* \& \* \& \n \& \* \& \n \& \n \\
        \n \& \* \& \* \& \n \& \n \& \n \\
        \* \& \* \& \n \& \n \& \n \& \n \\
        \* \& \n \& \n \& \n \& \n \& \n \\
      };
      \node[anchor = south east] (p-0-0) at (p-2-1.west |- p-1-1.north west) {};
      \foreach[count = \i] \v in {\(\Int{1, 2}\), \(\Int{1, 3}_1\), \(\Int{1, 3}_2\), \(\Int{1, 3}_3\), \(\Int{2, 2}_1\), \(\Int{2, 2}_2\), \(\Int{1, 4}\), \(\Int{2, 3}\), \(\Int{2, 4}\), \(\Int{3, 3}\), \(\Int{3, 4}\), \(\Int{4, 4}\)}{
        \node[left=2.5ex] at (p-\i-1) {\scriptsize\v};
      }
      \foreach[count = \i] \v in {\(\Int{4, 4}\), \(\Int{3, 4}\), \(\Int{3, 3}\), \(\Int{2, 4}\), \(\Int{2, 3}\), \(\Int{1, 4}\)}{
        \node[above=1.5ex] at (p-1-\i) {\scriptsize\v};
      }
    }
  \end{equation*}

\item Procedures on column $\Int{2, 4}$ provide us with the direct summands with dimension vectors
  $\dimv{0111}{0111}$, 
  $\dimv{1210}{0110}$, 
  $\dimv{1210}{0221}$, 
  $\dimv{1100}{0111}$, 
  $\dimv{1100}{0000}$, 
  $\dimv{0000}{0111}$, 
  $\dimv{0100}{0111}$, 
  $\dimv{1111}{0111}$, 
  $\dimv{1211}{0111}$, and
  $\dimv{0100}{0000}$. We note that $\dimv{1210}{0221}$ comes from a direct summand with three identity blocks. After all the operations, the matrix form is now:
  \begin{equation*}
    \mattikz{
      \matheader{
        \n \& \0 \& \* \& \0 \& \0 \& \0 \& \0 \& \E \\
        \n \& \0 \& \* \& \0 \& \0 \& \0 \& \E \& \0 \\
        \n \& \0 \& \* \& \0 \& \0 \& \E \& \0 \& \0 \\
        \n \& \* \& \* \& \0 \& \0 \& \0 \& \0 \& \0 \\
        \n \& \n \& \n \& \0 \& \0 \& \E \& \0 \& \n \\
        \0 \& \* \& \n \& \E \& \0 \& \0 \& \0 \& \0 \\
        \* \& \* \& \n \& \0 \& \0 \& \0 \& \0 \& \0 \\
        \n \& \0 \& \* \& \E \& \0 \& \0 \& \0 \& \n \\
        \n \& \0 \& \* \& \0 \& \E \& \0 \& \0 \& \n \\
        \n \& \* \& \* \& \0 \& \0 \& \0 \& \0 \& \n \\
        \* \& \* \& \n \& \0 \& \0 \& \0 \& \0 \& \n \\
        \n \& \* \& \* \& \n \& \n \& \n \& \n \& \n \\
        \* \& \* \& \n \& \n \& \n \& \n \& \n \& \n \\
        \* \& \n \& \n \& \n \& \n \& \n \& \n \& \n \\
      };
      \node[anchor = south east] (p-0-0) at (p-2-1.west |- p-1-1.north west) {};
      \foreach[count = \i] \v in {\(\Int{1, 3}_1\), \(\Int{1, 3}_2\), \(\Int{1, 3}_3\), \(\Int{1, 3}_4\), \(\Int{2, 2}\), \(\Int{1, 4}_1\), \(\Int{1, 4}_2\), \(\Int{2, 3}_1\), \(\Int{2, 3}_2\), \(\Int{2, 3}_3\), \(\Int{2, 4}\), \(\Int{3, 3}\), \(\Int{3, 4}\), \(\Int{4, 4}\)}{
        \node[left=2.5ex] at (p-\i-1) {\scriptsize\v};
      }
      \foreach[count = \i] \v in {\(\Int{4, 4}\), \(\Int{3, 4}\), \(\Int{3, 3}\), \(\Int{2, 4}_1\), \(\Int{2, 4}_2\), \(\Int{2, 4}_3\), \(\Int{2, 4}_4\), \(\Int{1, 4}\)}{
        \node[above=1.5ex] at (p-1-\i) {\scriptsize\v};
      }
    }.
  \end{equation*}

\item In the course of procedures on column $\Int{3, 3}$, we detect an erasable identity block.  This occurs  immediately after we transform $(\Int{2, 3}_1, \Int{3, 3})$ to Smith
  normal form. Prior to this some other blocks in column $\Int{3, 3}$ are processed without issue. Below, we show the relevant part of the current status of the matrix before and after taking Smith normal form:
  \begin{equation*}
    \mattikz{
      \matheader{
        \n \& \E \& \0 \\
        v_* \& \E \& \0 \\
        \E \& \0 \& \E \\
      };
      \node[anchor = south east] (p-0-0) at (p-2-1.west |- p-1-1.north west) {};
      \foreach[count = \i] \v in {
        \(\Int{1, 4}_1\), \(\Int{2, 3}_1\),\(\Int{2, 3}_3\) }{
        \node[left=2.5ex] at (p-\i-1) {\scriptsize\v};
      }
      \foreach[count = \i] \v in {\(\Int{3, 3}_1\), \(\Int{2, 4}_1\), \(\Int{2, 4}_3\)}{
        \node[above=1.5ex] at (p-1-\i) {\scriptsize\v};
      }
    }
    \cong
    \mattikz{
      \matheader{
        \n \& \n \& \E \& \0 \& \0 \& \0 \\
        \n \& \n \& \0 \& \E \& \0 \& \0 \\
        \E \& \0 \& \E \& \0 \& \0 \& \0 \\
        \0 \& \0 \& \0 \& \E \& \0 \& \0 \\
        \E \& \0 \& \0 \& \0 \& \E \& \0 \\
        \0 \& \E \& \0 \& \0 \& \0 \& \E \\
      };
      \node[anchor = south east] (p-0-0) at (p-2-1.west |- p-1-1.north west) {};
      \foreach[count = \i] \v in {\(\Int{1, 4}_1\), \(\Int{1, 4}_2\), \(\Int{2, 3}_1\), \(\Int{2, 3}_2\), \(\Int{2, 3}_3\), \(\Int{2, 3}_4\)}{
        \node[left=2.5ex] at (p-\i-1) {\scriptsize\v};
      }
      \foreach[count = \i] \v in {\(\Int{3, 3}_1\), \(\Int{3, 3}_2\), \(\Int{2, 4}_1\), \(\Int{2, 4}_2\), \(\Int{2, 4}_3\), \(\Int{2, 4}_4\)}{
        \node[above=1.5ex] at (p-1-\i) {\scriptsize\v};
      }
    }.
  \end{equation*}
  Note that the other identity submatrices are split after updating the
  block partitioning.

  We see that the identity block $(\Int{2, 3}_3, \Int{3, 3}_1)$ is erasable by additions
  from the rows in $\Int{2, 3}_1$ to the rows in $\Int{2, 3}_3$ and then
  additions from the columns in $\Int{2, 4}_3$ to the columns in
  $\Int{2, 4}_1$ to zero out the side effects.

  Here
  $\dimv{0010}{0010}$, 
  $\dimv{0110}{0010}$, 
  $\dimv{0110}{0111}$, 
  $\dimv{1221}{0121}$, 
  $\dimv{1110}{0010}$, 
  $\dimv{1210}{0111}$, 
  $\dimv{1210}{0121}$, 
  $\dimv{1110}{0111}$, 
  $\dimv{1110}{0121}$, 
  $\dimv{1110}{1111}$, 
  $\dimv{1110}{1121}$, 
  $\dimv{0000}{0010}$, 
  $\left( \dimv{1221}{0121} \right)$,
  $\left( \dimv{0110}{0111} \right)$,
  $\dimv{1220}{0121}$, and
  $\dimv{0110}{0121}$
  are extracted.
  We remark that indecomposables with dimension vectors $\dimv{1221}{0121}$ and $\dimv{0110}{0111}$ are extracted a second time. This is not double counting. Rather, the corresponding multiplicities are simply added together to get the correct multiplicities for these indecomposables.

  \begin{equation*}
    \mattikz{
      \matheader{
        \n \& \* \& \0 \\
        \0 \& \* \& \E \\
        \* \& \* \& \0 \\
        \n \& \0 \& \E \\
        \n \& \* \& \0 \\
        \* \& \* \& \0 \\
        \n \& \* \& \n \\
        \* \& \* \& \n \\
        \* \& \n \& \n \\
      };
      \node[anchor = south east] (p-0-0) at (p-2-1.west |- p-1-1.north west) {};
      \foreach[count = \i] \v in {\(\Int{1, 3}\), \(\Int{1, 4}_1\), \(\Int{1, 4}_2\), \(\Int{2, 3}_1\), \(\Int{2, 3}_2\), \(\Int{2, 4}\), \(\Int{3, 3}\), \(\Int{3, 4}\), \(\Int{4, 4}\)}{
        \node[left=2.5ex] at (p-\i-1) {\scriptsize\v};
      }
      \foreach[count = \i] \v in {\(\Int{4, 4}\), \(\Int{3, 4}\), \(\Int{2, 4}\)}{
        \node[above=1.5ex] at (p-1-\i) {\scriptsize\v};
      }
    }
  \end{equation*}

\item By procedures on column $\Int{3, 4}$, we extract
  $\dimv{0011}{0011}$,
  $\dimv{0111}{0011}$,
  $\dimv{0010}{0000}$,
  $\dimv{0110}{0000}$,
  $\dimv{1111}{0011}$,
  $\dimv{1221}{0111}$,
  $\dimv{1221}{0122}$,
  $\dimv{1110}{0011}$,
  $\dimv{1110}{0000}$,
  $\dimv{0000}{0011}$,
  $\dimv{0110}{0011}$,
  $\dimv{1221}{0011}$,
  $\dimv{0121}{0011}$,
  $\dimv{1121}{0011}$, and 
  $\dimv{0010}{0011}$; and leave the form:
  \begin{equation*}
    \mattikz{
      \matheader{
        \* \\
        \* \\
        \* \\
        \* \\
      };
      \node[anchor = south east] (p-0-0) at (p-2-1.west |- p-1-1.north west) {};
      \foreach[count = \i] \v in {\(\Int{1, 4}\), \(\Int{2, 4}\), \(\Int{3, 4}\), \(\Int{4, 4}\)}{
        \node[left=2.5ex] at (p-\i-1) {\scriptsize\v};
      }
      \foreach[count = \i] \v in {\(\Int{4, 4}\)}{
        \node[above=1.5ex] at (p-1-\i) {\scriptsize\v};
      }
    }.
  \end{equation*}
  
\item Finally, straightforward procedures on column $\Int{4, 4}$ yield
  $\dimv{0001}{0001}$,
  $\dimv{0001}{0000}$,
  $\dimv{0011}{0001}$,
  $\dimv{0011}{0000}$,
  $\dimv{0111}{0001}$,
  $\dimv{0111}{0000}$,
  $\dimv{1111}{0001}$,
  $\dimv{1111}{0000}$, and
  $\dimv{0000}{0001}$.
\end{enumerate}


\section{Discussion}

Let us show an example of how the proof breaks down in the representation infinite case ($n \geq 5$), in particular the case $\CL_5(ffff)$. We only need to consider the subproblem of the block matrix problem spanned by the rows $\Int{1, 4}, \Int{2, 3}$ and columns $\Int{2, 5}, \Int{3, 4}$:
\begin{equation*}
	\mattikz{
		\matheader{
      \* \& \* \\
      \* \& \* \\
		};
		\node[anchor = south east] (p-0-0) at (p-2-1.west |- p-1-1.north west) {};
		\foreach[count = \i] \v in {\(\Int{1, 4}\), \(\Int{2, 3}\)}{
      \node[left=3.5ex] at (p-\i-1) {\scriptsize\v};
		}
		\foreach[count = \i] \v in {\(\Int{3, 4}\), \(\Int{2, 5}\)}{
      \node[above=1.5ex] at (p-1-\i) {\scriptsize\v};
    }
    \node[right=6pt] (LR) at (p-1-2.east |- p-2-2.south east) {};
     \draw[<->, red, postaction={decorate}] (p-2-2.south |- LR) to[out=240,in=300] (p-2-1.south |- LR);
     \draw[<->, red, postaction={decorate}] (LR |- p-2-2.east)  to[out=30,in=330] (LR |- p-1-2.east);
	}.
\end{equation*}
Note that operations between rows
$\Int{1, 4}$ and $\Int{2, 3}$ are impermissible, and likewise so are
operations between columns $\Int{2, 5}$ and $\Int{3, 4}$. After applying the procedure, we eventually obtain
\begin{equation*}
	\begin{tikzpicture}[baseline = (p.center)]
		\matheader{
		A & E \\
		E & E \\
		};
		\node[anchor = south east] (p-0-0) at (p-2-1.west |- p-1-1.north west) {};
		\foreach[count = \i] \v in {\(\Int{1, 4}\), \(\Int{2, 3}\)}{
		\node[left=3.5ex] at (p-\i-1) {\scriptsize\v};
		}
		\foreach[count = \i] \v in {\(\Int{3, 4}\), \(\Int{2, 5}\)}{
		\node[above=1.5ex] at (p-1-\i) {\scriptsize\v};
		}
	\end{tikzpicture}
\end{equation*}
as a subproblem, where $A$ is an
unerasable unprocessed submatrix and the $E$'s are processed identity matrices of
the same size.

The procedure breaks down when we try to reduce the unprocessed block
$A$ to Smith normal form. Starting with row operations on $\Int{1,4}$,
we get side effects on the neighboring block. Fixing those, we get side effects on the block below, and so on:
\begin{equation*}
	\begin{tikzpicture}[baseline = (p.center)]
		\matheader{
		PA & P \\
		E & E \\
		};
		\node[anchor = south east] (p-0-0) at (p-2-1.west |- p-1-1.north west) {};
		\foreach[count = \i] \v in {\(\Int{1, 4}\), \(\Int{2, 3}\)}{
		\node[left=3.5ex] at (p-\i-1) {\scriptsize\v};
		}
		\foreach[count = \i] \v in {\(\Int{3, 4}\), \(\Int{2, 5}\)}{
		\node[above=1.5ex] at (p-1-\i) {\scriptsize\v};
		}
	\end{tikzpicture}
  \cong
	\begin{tikzpicture}[baseline = (p.center)]
		\matheader{
		PA & E \\
		E & P^{-1} \\
		};
		\node[anchor = south east] (p-0-0) at (p-2-1.west |- p-1-1.north west) {};
		\foreach[count = \i] \v in {\(\Int{1, 4}\), \(\Int{2, 3}\)}{
		\node[left=3.5ex] at (p-\i-1) {\scriptsize\v};
		}
		\foreach[count = \i] \v in {\(\Int{3, 4}\), \(\Int{2, 5}\)}{
		\node[above=1.5ex] at (p-1-\i) {\scriptsize\v};
		}
	\end{tikzpicture}
  \cong
  \begin{tikzpicture}[baseline = (p.center)]
		\matheader{
		PA & E \\
		P & E \\
		};
		\node[anchor = south east] (p-0-0) at (p-2-1.west |- p-1-1.north west) {};
		\foreach[count = \i] \v in {\(\Int{1, 4}\), \(\Int{2, 3}\)}{
		\node[left=3.5ex] at (p-\i-1) {\scriptsize\v};
		}
		\foreach[count = \i] \v in {\(\Int{3, 4}\), \(\Int{2, 5}\)}{
		\node[above=1.5ex] at (p-1-\i) {\scriptsize\v};
		}
	\end{tikzpicture}
  \cong
	\begin{tikzpicture}[baseline = (p.center)]
		\matheader{
		PAP^{-1} & E \\
		E & E \\
		};
		\node[anchor = south east] (p-0-0) at (p-2-1.west |- p-1-1.north west) {};
		\foreach[count = \i] \v in {\(\Int{1, 4}\), \(\Int{2, 3}\)}{
		\node[left=3.5ex] at (p-\i-1) {\scriptsize\v};
		}
		\foreach[count = \i] \v in {\(\Int{3, 4}\), \(\Int{2, 5}\)}{
		\node[above=1.5ex] at (p-1-\i) {\scriptsize\v};
		}
	\end{tikzpicture}.
\end{equation*}

Next, we comment on existing methods and the computational aspects of our work.
Here, we take the persistence module as input and measure computational cost with respect to its size. Requiring the explicit computation of the persistence module $H_q(\mathbb{X})$ may be wasteful. Indeed, classical algorithms for persistence work directly on the level of diagrams of simplicial complexes (for example \cite{comput_ph,zigzag_socg} for filtrations or zigzags, respectively). It is of interest to find similar algorithms that compute an indecomposable decomposition of $H_q(\mathbb{X})$ directly, without explicit computation of the persistence module.

The problem of computing an indecomposable decomposition of a module $M$ over a finite dimensional $K$-algebra $A$ has been well-studied \cite{chistov,lux}. For example, assuming that $K$ is a finite field, \cite{chistov} shows that there is a polynomial time algorithm to find an indecomposable decomposition of $M$.
First, one computes the endomorphism algebra $\End_A(M)=\Hom(M,M)$. Then, a complete set of primitive orthogonal idempotents $\pi_1,\hdots,\pi_\ell$ is computed. The submodules $\pi_i(M)$, for $i=1,\hdots,\ell$, provides a decomposition of $M$ into indecomposables.

One straightforward way to compute $\End_A(M)$ is by solving a system of linear equations.
Let $M=(M_a,\varphi_\alpha)_{a\in Q_0,\alpha\in Q_1}$ be a representation of a bound quiver $(Q,P)$, $d_a=\dim M_a$, and $\hat{d}=\max_{a\in Q_0}{d_a}$. We assume that we have fixed bases for $M_a$, and the actions $\varphi_\alpha$ are given in terms of $K$-matrices. For each vertex $a\in Q_0$, we set up a $d_a\times d_a$ $K$-matrix of unknowns $X_a$, and solve the condition given in Eq.~\eqref{eq:repn_hom} that $X_b\varphi_\alpha - \varphi_\alpha X_a = 0$ for each arrow $\alpha: a \rightarrow b$, giving a linear system of $\displaystyle\sum_{a\in Q_0}d_a^2=\mathcal{O}(\hat{d}^2)$ unknowns and $\displaystyle\sum_{(\alpha:a\rightarrow b)\in Q_1}d_ad_b=\mathcal{O}(\hat{d}^2)$ equations. Thus, it seems to require $\mathcal{O}(\hat{d}^6)$  computational time, even before the rest of the computations (finding primitive orthogonal idempotents and computing images).
There are more clever methods of computing the endomorphism rings, for example by reducing the number of equations and unknowns by first computing a generating system for the module \cite{lux2} or by using Gr\"{o}bner basis methods \cite{green}.

Finally, let us provide an estimate of the computational complexity of our procedure. First, we recapitulate the matrix formalism in Subsection~\ref{subsec:matrixformalism}.
Given a persistence module $M$ over a commutative ladder with $n<5$, writing it as an arrow $F(M):V\rightarrow W$ is straightforward. Then, we decompose its upper and lower rows while keeping track of basis changes to get the forms of the matrices of arrow $\Phi$ relative to the new basis, costing $\mathcal{O}(\hat{d}^3)$.

Let us analyze the size of the corresponding block matrix problem of $M$. Recall that the blocks are of size $m_{c,d}' \times m_{a,b}$ for $\intv[a,b] \reltoeq \intv[c,d]$, where the numbers $m_{a,b}$ and $m'_{c,d}$ are determined by the decomposition of $F(M)$ as in Eq.~\eqref{eq:decompose_cl_rows}. It is clear that $m_{a,b}$ and $m'_{c,d}$ are bounded above by $\hat{d}$, for otherwise, there will be a vertex $x$ of the commutative ladder with $d_x > \hat{d}$, a contradiction. We also compute:
\[
  n\hat{d} \geq \dim V = \sum_{1\leq a\leq b \leq n} \dim \intv[a,b]^{m_{a,b}} = \sum_{1\leq a\leq b \leq n} m_{a,b}(b-a+1) \geq \sum_{1\leq a\leq b \leq n} m_{a,b}
\]
from Eq.~\eqref{eq:decompose_cl_v}, and a similar estimate for the sum of the $m'_{c,d}$. Thus, the total size of the block matrix problem is bounded above by $n\hat{d}$.

The above analysis considers the size of the block matrix problem by counting actual rows and columns of $K$-scalars. In contrast, the number of blocks $b$ themselves is not dependent on the dimension $\hat{d}$, but rather depends only on the orientation $\tau$ and what operations have been performed so far.

Next, we look at Algorithm~\ref{algo:main}. Each query for side effects in Algorithm~\ref{algo:sideeffects} costs $\mathcal{O}(b)$. By the proof of Theorem~\ref{main_theorem}, Lines~\ref{line:main_colfix} and \ref{line:main_rowfix} of Algorithm~\ref{algo:main} do not lead to infinite recursion, thus only a finite number of such queries and matrix operations are needed for Algorithm~\ref{algo:colrowfix}. This number depends not on $\hat{d}$ but only on the current arrangement of block statuses of the block matrix problem.

Furthermore, the check for erasbility and construction of a process tree $p$ by \textproc{erasable} in Line~\ref{line:main_erasable} of Algorithm~\ref{algo:main} does not depend on $\hat{d}$. Indeed, the functions \textproc{row\_erasable} and \textproc{col\_erasable} does not perform any matrix operations. Rather, they simply query the arrangement of matrix statuses and permissible operations in the block matrix problem.

The actual operations on the block matrix problem are given by: transforming a block to Smith normal form, fixing one side effect, zeroing out $v_t$ by $v_E$ for one vertex $(v_t, v_E)$ in the process tree $p$. Each of these cost $\mathcal{O}(n\hat{d}^3)$. Thus, the total cost is simply the number of operations $C$ needed, which is not dependent on $\hat{d}$, times $\mathcal{O}(n\hat{d}^3)$, yielding an overall estimate of $\mathcal{O}(Cn\hat{d}^3)$.


\vspace{1em}
\small{
  \noindent\textbf{Acknowledgements: }
  On behalf of all authors, the corresponding author states that there is no conflict of interest.
  This work was partially supported by JST CREST Mathematics 15656429.
  H.T. is supported by JSPS KAKENHI Grant Number JP16J03138.
}


\end{document}